\documentclass[final]{amsart}

\usepackage{amsmath}
\usepackage{amssymb}

\usepackage{nicefrac} 
\usepackage{mathtools} 
\usepackage{enumitem} 
\setlist[enumerate,1]{label={\textup{(\roman*)}}} 
\usepackage{thmtools, thm-restate} 
\usepackage[hidelinks]{hyperref} 

\theoremstyle{proclaim}
\newtheorem{theorem}{Theorem}[section]
\newtheorem{question}[theorem]{Question}
\newtheorem{problem}[theorem]{Problem}
\newtheorem{lemma}[theorem]{Lemma}
\newtheorem{corollary}[theorem]{Corollary}
\newtheorem{proposition}[theorem]{Proposition}

\newtheorem{algorithm}[theorem]{Algorithm}

\theoremstyle{statement}
\newtheorem{definition}[theorem]{Definition}
\newtheorem{remark}[theorem]{Remark}

\numberwithin{equation}{section}

\DeclareMathOperator{\trace}{Tr} 
\DeclareMathOperator{\diag}{diag}
\DeclareMathOperator{\rank}{rank}
\DeclareMathOperator{\range}{ran}
\DeclareMathOperator{\spec}{sp}
\DeclareMathOperator{\id}{Id}
\DeclareMathOperator{\spans}{span}

\renewcommand{\Re}{\ensuremath{\mathfrak{Re}}}
\renewcommand{\Im}{\ensuremath{\mathfrak{Im}}}

\newcommand{\norm}[1]{\ensuremath{\left\lVert #1 \right\rVert }}
\newcommand{\abs}[1]{\ensuremath{\left\lvert #1 \right\rvert}}
\newcommand{\angles}[1]{\ensuremath{\left\langle #1 \right\rangle}}

\newcommand{\snorm}[1]{\ensuremath{\lVert #1 \rVert }}
\newcommand{\sabs}[1]{\ensuremath{\lvert #1 \rvert}}

\expandafter\mathchardef\expandafter\varphi\number\expandafter\phi\expandafter\relax
\expandafter\mathchardef\expandafter\phi\number\varphi

\begin{document}




\title[Diagonality of idempotents]{Diagonality and Idempotents \\ with applications to problems in \\ operator theory and frame theory}
\author[Jireh Loreaux {\protect \and} Gary Weiss]{Jireh Loreaux {\protect \and} Gary Weiss}
\address{Jireh Loreaux, Department of Mathematics, McMicken College of Arts and Sciences, University of Cincinnati, Cincinnati, Ohio, 45212, USA}
\email{loreaujy@mail.uc.edu, loreaujy@gmail.com}
\address{Gary Weiss, Department of Mathematics, McMicken College of Arts and Sciences, University of Cincinnati, Cincinnati, Ohio, 45212, USA}
\email{gary.weiss@uc.edu, weissg@ucmail.uc.edu, gary.weiss@math.uc.edu}

\begin{abstract}
  We prove that a nonzero idempotent is zero-diagonal if and only if it is not a Hilbert--Schmidt perturbation of a projection, along with other useful equivalences.
  Zero-diagonal operators are those whose diagonal entries are identically zero in some basis.
  
  We also prove that any bounded sequence appears as the diagonal of some idempotent operator, thereby providing a characterization of inner products of dual frame pairs in infinite dimensions. 
  Furthermore, we show that any absolutely summable sequence whose sum is a positive integer appears as the diagonal of a finite rank idempotent.
\end{abstract}



\maketitle


\section{INTRODUCTION}

Throughout this paper the underlying space is either a finite dimensional or separable infinite dimensional complex Hilbert space.
We first establish some terminology. 
A ``basis'' herein is an \emph{orthonormal} basis of the underlying Hilbert space. 
When a basis $\mathfrak{e} = \{e_j\}_{j=1}^N$ (possibly $N=\infty$) is specified, ``diagonal'' or ``diagonal sequence'' of an operator $T$ is the sequence $\angles{(Te_j,e_j)}$, that is, the diagonal sequence of the matrix representation for $T$ with respect to the basis $\mathfrak{e}$. 
Sometimes we will say that a sequence is ``a diagonal'' of $T$, by which we mean that there exists some basis with respect to which $T$ has this sequence as its diagonal.

Diagonality is a term coined by the authors for the study of:
\begin{enumerate}[label=(\alph*)]

\item \label{item:4} properties that the diagonal sequences can possess for a fixed operator and in all bases, and characterizations of those sequences;

\item \label{item:5} properties that the diagonal sequences can possess for a class of operators and in all bases, and characterizations of those sequences. 
\end{enumerate}
Such information is used ubiquitously throughout operator theory. 
With this term we here attempt to bring 
these phenomena under a unifying umbrella in the hope this will stimulate bridges of insight connecting them.
This paper focuses mainly on \ref{item:5}, but some results also have the flavor of \ref{item:4}.

Starting with the most basic and then on to current active areas of research, we give some examples that pervade our work.

\begin{enumerate}[leftmargin=0cm,itemindent=0.75cm,labelwidth=\itemindent,labelsep=0cm,align=left, topsep=6pt, itemsep=6pt]

\item Which numbers can appear on the diagonal of an operator? 
Clearly these numbers constitute precisely its numerical range.
And which operators have only positive diagonal entries? Clearly these are the positive operators.

\item Well-known highly useful diagonality example: 
every trace-class operator in every basis has an absolutely summable diagonal sequence and those sums are invariant; likewise every compact operator has diagonal sequences tending to zero in every basis.
 In contrast, finite rank operators fail to always have finite rank diagonals, witness any nonzero rank-one projection $\xi\otimes\xi$, $\xi\in\ell^2$ of infinite support. 

This phenomenon for the trace-class ideal and the ideal of compact operators is subsumed under the more general notion of diagonal invariance.
 Given a basis $\mathfrak{e}$, we let $E_{\mathfrak{e}}(T)$ be the conditional expectation of $T$ with respect to $\mathfrak{e}$ which replaces the off-diagonal entries with zeros.
 An ideal $\mathcal{I}$ is said to be diagonally invariant if for every $\mathfrak{e}$ and every $T\in\mathcal{I}$, $E_{\mathfrak{e}}(T)\in\mathcal{I}$.
 Diagonal invariance is equivalent to the ideal being arithmetic mean-closed ($_a(\mathcal{I}_a) = \mathcal{I}$, am-closed for short). For details see \cite[Theorem 4.5]{KW-2011-IUMJ}, but for now, $\mathcal{I}_a$ and $_a\mathcal{I}$ are the arithmetic mean and pre-arithmetic mean ideals generated respectively by: operators with $s$-numbers the arithmetic means of the $s$-numbers of operators from ideal $\mathcal{I}$, and operators whose arithmetic means of their $s$-numbers are $s$-numbers of operators in $\mathcal{I}$.

The converses seem to us to be less well-known: if in every basis an operator's diagonal sequence is absolutely summable, then the operator is trace-class; and likewise if in every basis the operator's diagonal sequence tends to zero, then it is a compact operator.
This phenomenon is totally general.
That is, a sufficient test for membership in an arbitrary ideal $\mathcal{I}$ is: 
\begin{equation}
  \label{eq:2}
  E_{\mathfrak{e}}(T)\in\mathcal{I}, \forall\mathfrak{e} \implies T\in\mathcal{I}.  
\end{equation}
Although not immediate, this follows easily from the contrapositive by considering the real and imaginary parts of $T$, and by considering separately the compact and non-compact cases. 

\item \label{item:6} What diagonal sequences can arise for a specific operator? The study of \ref{item:4}. 

  We think of this subject as Schur--Horn theory, although traditionally Schur--Horn theory refers to the study of the diagonals of selfadjoint operators, a study almost a century old that continues today and is beginning to extend into operator algebras.

Much of this work focuses on diagonals of positive compact operators.
A fundamental tool used is majorization theory, including new types of majorization such as $\infty$- and approximate $\infty$-majorization defined using $p$- and approximate $p$-majorization.
 Convexity also plays a central role.
Some 1923--1964 contributors are Schur \cite{Sch-1923-SBMG}, Horn \cite{Hor-1954-AJoM}, Markus \cite{Mar-1964-UMN}, Gohberg--Markus \cite{GM-1964-MSN}, and in the last 10 years --- Arveson--Kadison \cite{AK-2006-OTOAaA}, Antezana--Massey--Ruiz--Stojanoff \cite{AMRS-2007-IJoM}, Kaftal--Weiss \cite{KW-2010-JoFA} and  Loreaux--Weiss \cite{LW-2015-JFA}.
 Others for operator algebra Schur--Horn theory include Argerami and Massey \cite{AM-2007-IUMJ}, \cite{AM-2008-JMAA}, \cite{AM-2013-PJM} and most recently Ravichandran \cite{Rav-2012} and Kennedy--Skoufranis \cite{KS-2014}. 

Schur--Horn theory for finite spectrum selfadjoint operators was studied extensively by Kadison \cite{Kad-2002-PotNAoS}, \cite{Kad-2002-PotNAoSa} (the carpenter problem for projections, or equivalently 2-point spectrum normal operators), Arveson \cite{Arv-2007-PotNAoSotUSoA} (a necessary condition on diagonals of certain finite spectrum normal operators), Jasper \cite{Jas-2013-JoFA} (3-point spectrum selfadjoint operators), and Bownik--Jasper \cite{BJ-2012}, \cite{BJ-2013-TAMS} (finite spectrum selfadjoint operators), and along with \cite{Neu-1999-JoFA} are the only non-compact operator results known to the authors.

In \cite{Neu-1999-JoFA}, A. Neumann obtained a Schur--Horn type theorem for general selfadjoint operators.
 However, it should be noted that his results are \emph{approximate} in the sense that he identified the $\ell^\infty$-closure of the diagonal sequences of a selfadjoint operator with a certain convex set. 
In contrast, the aforementioned results of Kaftal--Weiss, Loreaux--Weiss, Kadison, Jasper and Bownik--Jasper are all \emph{exact} in the sense that they describe precisely the diagonals of certain classes of selfadjoint operators. 

\item \label{item:7} What diagonal sequences can arise for a class of operators? The study of \ref{item:5}.

There is a variety of material on this subject.
We reference only that which we know, but there are almost certainly results we have inadvertently overlooked.

In the same paper \cite{Hor-1954-AJoM} in which he characterizes the diagonal sequences of a fixed selfadjoint matrix in $M_n(\mathbb{C})$, Horn identifies the diagonals of the class of rotation matrices. 
He then uses this to identify the diagonals of the classes of orthogonal matrices and of unitary matrices. See \cite[Theorems 8-11]{Hor-1954-AJoM}.

Fong shows in \cite{Fon-1986-PotEMSSI} that any bounded sequence of complex numbers appears as the diagonal of a nilpotent operator in $B(H)$ of order four ($N^4=0$), thus seamlessly characterizing diagonals of the broader classes of nilpotent and also quasinilpotent operators. 
In this paper Fong remarks that a finite complex-valued sequence appears as the diagonal of a nilpotent matrix in $M_n(\mathbb{C})$ if and only if its sum is zero. 

More recently, Giol, Kovalev, Larson, Nguyen and Tener \cite{GKL+-2011-OaM} classified the diagonals of idempotent matrices in $M_n(\mathbb{C})$ as those whose sum is a positive integer less than $n$, along with the constant sequences $\langle 0,\ldots,0 \rangle$ and $\langle 1,\ldots,1 \rangle$ (see Theorem~\ref{thm:idempotent-matrix-diagonals} below).

\item In this context, J.~Jasper posed to us a frame theory question which for us evolved into questions below on diagonal sequences of idempotents (operators for which $D^2=D$) and gave rise to this paper: Questions~\ref{que:ell1-diagonal-finite-rank}--\ref{que:zero-diagonal-finite-rank} below and the immediately preceding comment on the frame theory connection.
\end{enumerate}

As mentioned above, a good deal of work concerning diagonal sequences of operators deals with the selfadjoint case.
Here we study diagonal sequences of idempotents, and so diagonals of projections (selfadjoint idempotents) are of particular relevance to us. 
These were characterized by Kadison in \cite{Kad-2002-PotNAoS}, \cite{Kad-2002-PotNAoSa} in the following theorem. 
We find this theorem especially interesting because it straddles the fence between \ref{item:6} and \ref{item:7}.
Indeed, although it is stated as a characterization of the diagonals of the class of projections, it can easily be adapted to identify the diagonals of any fixed projection.
This is because two projections $P,P'\in B(H)$ are unitarily equivalent if and only if $\trace P = \trace P'$ and $\trace (1-P) = \trace (1-P')$.
And so for $\angles{d_k}$ an admissible diagonal sequence for $P$, these trace quantities are precisely the sum of the diagonal entries $d_k$ and the sum of $1-d_k$, respectively. 
Then one can apply the four finite/infinite cases in the next theorem.

\begin{theorem}[\protect{\cite{Kad-2002-PotNAoS}}, \protect{\cite{Kad-2002-PotNAoSa}}]
  \label{thm:projection-diagonals}
  Given an infinite sequence $\langle d_k \rangle\in [0,1]^{\mathbb{N}}$ with 
  \begin{equation*} 
  a = \sum_{d_k<\nicefrac{1}{2}} d_k \quad \text{and} \quad b = \sum_{d_k\ge \nicefrac{1}{2}} (1-d_k),
  \end{equation*}
  then there is a projection $P\in B(H)$ (i.e., $P^2=P=P^*$) with diagonal $\langle d_k \rangle$ if and only if one of the following mutually exclusive conditions holds:
  \begin{enumerate}
  \item either $a$ or $b$ is infinite;
  \item $a,b<\infty$ and $a-b\in\mathbb{Z}$.
  \end{enumerate}
\end{theorem}
The requirement that $0\le d_k\le 1$ for all $k\in\mathbb{N}$ is clearly necessary since $P\ge 0$, $\snorm{P}=1$ and the diagonal entries of $P$ are elements of its numerical range. 
The second condition, that $a-b\in \mathbb{Z}$, is less obvious but can viewed as a kind of index obstruction to an arbitrary sequence in $[0,1]^{\mathbb{N}}$ appearing as the diagonal of a projection. 
Indeed, in \cite{Arv-2007-PotNAoSotUSoA}, Arveson provided details on this index obstruction and showed that it applies more generally to any normal operator with finite spectrum that consists of the vertices of a convex polygon. 

Since we study diagonals of idempotents in $B(H)$, which when not projections are non-selfadjoint, we are interested in diagonals of non-selfadjoint operators.
 One particularly relevant result in this direction is the aforementioned characterization of diagonals of idempotent matrices in $M_n(\mathbb{C})$ by Giol, Kovalev, Larson, Nguyen and Tener \cite{GKL+-2011-OaM}. 

\begin{theorem}[\protect{\cite[Theorem 5.1]{GKL+-2011-OaM}}]
  \label{thm:idempotent-matrix-diagonals}
  A finite sequence $\langle d_k \rangle\in\mathbb{C}^n$ appears as the diagonal of an idempotent $D\in M_n(\mathbb{C})$ if and only if one of the following three mutually exclusive conditions holds.
  \begin{enumerate}
  \item $d_k=0$ for all $k$ (in which case $D=0$);
  \item $d_k=1$ for all $k$ (in which case $D=I$);
  \item $\sum d_k \in \{1,\ldots,n-1\}$.
  \end{enumerate}
\end{theorem}
Since $\trace D\in\{1,\ldots,n-1\}$ for any nonzero, non-identity idempotent matrix (as is well-known, see for instance Lemma~\ref{lem:idempotent-decomp}), this theorem says that this is the only requirement for a sequence to appear as the diagonal of some idempotent. 

Giol, Kovalev, Larson, Nguyen and Tener were interested in this result because of its relevance to frame theory. 
Because of a similar frame-theoretic question (characterizing inner products of dual frame pairs) Jasper asked for a characterization of diagonals of idempotents in $B(H)$. 
Such a result would simultaneously be an extension of the previous two theorems. 
For a key test case, Jasper posed to us the following two operator-theoretic questions (private communication, May 2013 \cite{Jas-2013}):

\begin{question}
  \label{que:ell1-diagonal-finite-rank}
  If an idempotent has a basis in which its diagonal is absolutely summable, is it finite rank?
\end{question}

\begin{question}
  \label{que:zero-diagonal-finite-rank}
  If an idempotent has a basis in which its diagonal consists solely of zeros (i.e., is a zero-diagonal operator in the terminology of Fan \cite{Fan-1984-TotAMS}), is it finite rank?
\end{question}

If we restrict the idempotents to be selfadjoint (i.e., projections), then they are positive operators and the answer to each question is certainly affirmative since the trace is preserved under conjugation by a unitary operator (i.e., a change of basis). 
In fact, for projections, having an absolutely summable (or even summable) diagonal is a characterization of those projections with finite rank since $\rank P = \trace P$.
Moreover, the only projection with a zero diagonal is the zero operator for this same reason.
Hence, a negative answer to either of these questions for the entire class of idempotents would be a notable departure from the case of projections, and would therefore suggest that the classification of their diagonals is potentially harder than one might na\"{i}vely expect. 

As it turns out, Larson constructed a nonzero (and even necessarily infinite rank) idempotent  that lies in a continuous nest algebra which has zero diagonal with respect to this nest \cite[Proof of Theorem 3.7]{Lar-1985-AoMSS}. 
An operator $T$ has zero diagonal \emph{with respect to the nest} if $P_{\lambda}TP_{\lambda}=0$ for some linearly ordered set of projections $\{P_{\lambda}\}_{\lambda\in\Lambda}$ inside the nest such that with respect to the decomposition of the identity $I=\sum_{\Lambda}P_{\lambda}$ every element of the nest is block upper-triangular.
However, the existence of an idempotent with zero diagonal with respect to a nest algebra certainly depends on the order type of the nest to some extent. 
For example, the nest algebra consisting of the upper triangular matrices with respect to some basis $\{e_n\}_{n\in\mathbb{N}}$ for $H$ has order type $\omega$ (the first infinite ordinal), and simple computations show that the only idempotent with zero diagonal inside this nest algebra is the zero operator. 

Once we leave the realm of nest algebras, we can ask two questions:
\begin{itemize}
\item Which idempotents are zero-diagonal?
\item Which idempotents have an absolutely summable diagonal?
\end{itemize} 
As it turns out, both of these questions have the same answer, which we provide in Theorem~\ref{thm:idempotent-zero-diagonal}. 
Before we state this theorem, we expound slightly on the methods involved. 

The techniques for analyzing diagonals of non-selfadjoint operators seem to differ greatly from those used for selfadjoint operators. 
For example, the techniques used in determining diagonals of selfadjoint operators often rely heavily on majorization and keeping track of the explicit changes of the basis (or equivalently, the unitary operators) involved in the construction. 
In contrast, the Toeplitz--Hausdorff Theorem, that the numerical range $W(T)$ of a bounded operator $T$ is convex, is one of the central tools in the work of Fan, Fong and Herrero \cite{Fan-1984-TotAMS}, \cite{FF-1994-PotAMS}, \cite{FFH-1987-PotAMS} to determine diagonals of non-selfadjoint operators. 
Indeed, they frequently use the nonconstructive version of the Toeplitz--Hausdorff Theorem despite the existence of constructive versions in which a formula is specified for the vector yielding the prescribed value of the quadratic form.

The Fan, Fong and Herrero results relevant to us here are restated below.
The first is an infinite dimensional analogue of the finite dimensional result that an $n\times n$ matrix has trace zero if and only if it is zero-diagonal. 

\begin{theorem}[\protect{\cite[Theorem 1]{Fan-1984-TotAMS}}]
  \label{thm:subsequence-zero-diagonalizability}
  If $T\in B(H)$ and there exists some basis $\{e_j\}_{j=1}^\infty$ for $H$ for which the partial sums
  \begin{equation*} s_n \coloneqq{} \sum_{j=1}^n (Te_j,e_j) \end{equation*}
have a subsequence converging to zero, then $T$ is zero-diagonal. 
\end{theorem}

\begin{definition}[\protect{\cite{FFH-1987-PotAMS}}]
  \label{def:trace-set}
  Let $T\in B(H)$ and let $\mathfrak{e}=\{e_j\}_{j=1}^\infty$ be a basis for $H$.
 Suppose the partial sums $s_n = \sum_{j=1}^n (Te_j,e_j)$ converge to some value $s\in\mathbb{C}$.
 Then we say that $\trace_{\mathfrak{e}} T:=s$ is the \emph{trace of $T$ with respect to the basis $\mathfrak{e}$}.
The \emph{set of traces of $T$}, denoted $R\{\trace T\}$, is then the set of all such traces $\trace_{\mathfrak{e}} T$ as $\mathfrak{e}$ ranges over all orthonormal bases for which $\trace_{\mathfrak{e}} T$ is defined.

Observe that in order to make sense of this definition it is essential both that these trace values are finite and that we must order $\mathfrak{e}$ by $\mathbb{N}$.
\end{definition}

A curious fact about the set $R\{\trace T\}$ from Definition~\ref{def:trace-set} is that it may take on only four different shapes: the plane, a line, a point or the empty set. 
It is no coincidence that these shapes coincide with those obtainable as the limits of convergent rearrangements of a series of complex numbers (i.e., the L\'{e}vy--Steinitz Theorem extending the Riemann Rearrangement Theorem to complex numbers). 

\begin{theorem}[\protect{\cite[Theorem 4]{FFH-1987-PotAMS}}]
  \label{thm:trace-range-shape}
  Suppose $T\in B(H)$. Then there are four possible shapes that $R\{\trace T\}$ can acquire. More specifically, $R\{\trace T\}$ is:
  \begin{enumerate}
  \item the plane $\mathbb{C}$ if and only if for all $\theta\in\mathbb{R}$, $(\Re\, \mathrm{e}^{\mathrm{i}\theta}T)_+\notin\mathcal{C}_1$ (the trace-class);
  \item a line if and only if for some $\theta\in\mathbb{R}$, \\ $(\Re\, \mathrm{e}^{\mathrm{i}\theta}T)_\pm\notin\mathcal{C}_1$ but $(\Im\, \mathrm{e}^{\mathrm{i}\theta}T)\in\mathcal{C}_1$;
  \item a point if and only if $T\in\mathcal{C}_1$;
  \item the empty set $\emptyset$ if and only if for some $\theta\in\mathbb{R}$,  \\ $(\Re\, \mathrm{e}^{\mathrm{i}\theta}T)_+\notin\mathcal{C}_1$ but $(\Re\, \mathrm{e}^{\mathrm{i}\theta}T)_-\in\mathcal{C}_1$.
  \end{enumerate}
\end{theorem}

In fact, their proof of Theorem~\ref{thm:trace-range-shape} shows that given $T\in B(H)$ there exists a basis $\mathfrak{e}$ for $H$, $\mathfrak{e}$ ordered by $\mathbb{N}$, for which every element of $R\{\trace T\}$ can be obtained from a basis which is a permutation of $\mathfrak{e}$. 
For the next theorem Fan--Fong utilize the previous two theorems to provide intrinsic (i.e., basis independent) criteria for when a bounded operator is zero-diagonal.

\begin{theorem}[\protect{\cite{FF-1994-PotAMS}}]
  \label{thm:zero-diagonal-trace-theta}
  An operator $T$ is zero-diagonal if and only if for all $\theta\in\mathbb{R}$,
\begin{equation*} \trace(\Re\, \mathrm{e}^{\mathrm{i}\theta}T)_+=\trace(\Re\, \mathrm{e}^{\mathrm{i}\theta}T)_-. \end{equation*}
\end{theorem}

We neither use nor cite this theorem elsewhere in the paper.
 However, it seems interesting to include it because it shares its intrinsic nature  with our Theorem~\ref{thm:idempotent-zero-diagonal}(i).

Later we will use Theorem~\ref{thm:trace-range-shape} to prove our first main theorem:

\begin{restatable*}{theorem}{firstmaintheorem}
  \label{thm:idempotent-zero-diagonal}
  For $D\in B(H)$ an infinite rank idempotent the following are equivalent:
  \begin{enumerate}
  \item \label{item:mthm1} $D$ is not a Hilbert--Schmidt perturbation of a projection;
  \item \label{item:mthm2} the nilpotent part of $D$ is not Hilbert--Schmidt;
  \item \label{item:mthm3} $R\{\trace D\}=\mathbb{C}$;
  \item \label{item:mthm4} $D$ is zero-diagonal;
  \item \label{item:mthm5} $D$ has an absolutely summable diagonal;
  \item \label{item:mthm6} $D$ has a summable diagonal (i.e., $R\{\trace D\}\not=\emptyset$).
  \end{enumerate}
\end{restatable*}

We have not yet defined the \emph{nilpotent part} of an idempotent $D$, but it is a natural object defined in Lemma~\ref{lem:idempotent-decomp} that gives a canonical decomposition for idempotents. 
It turns out that \ref{item:mthm5} and \ref{item:mthm6} of Theorem~\ref{thm:idempotent-zero-diagonal} are actually equivalent for \emph{any} bounded operator, not merely idempotents (see Proposition~\ref{prop:sum-iff-abs-sum-diag}).

Our next main theorem answers Jasper's frame theory problem which, as equivalently stated earlier, characterizes diagonals of the class of idempotents.
The equivalence of these two problems was originally described to us by Jasper, but a fairly concise explanation can be found on the MathOverflow post: \url{http://mathoverflow.net/q/132592}.

\begin{restatable*}{theorem}{secondmaintheorem}
  \label{thm:idempotent-diagonals-ell-infty}
  Every $\langle d_n \rangle\in\ell^\infty$ admits an idempotent $D\in B(H)$ whose diagonal is $\langle d_n \rangle$ with respect to a basis $\mathfrak{b}$. 
\end{restatable*}

Herein we use $\mathfrak{b}$ to denote a target basis, whereas we use $\mathfrak{e}$ to denote an arbitrary basis.
While Theorem~\ref{thm:idempotent-diagonals-ell-infty} can be viewed as an extension of Theorem~\ref{thm:idempotent-matrix-diagonals}, so also can our last main theorem.

\begin{restatable*}{theorem}{thirdmaintheorem}
  \label{thm:finite-rank-idempotent-diagonals}
  The diagonals of the class of nonzero finite rank idempotents consist precisely of those absolutely summable sequences whose sum is a positive integer. 
\end{restatable*}

\section{ZERO-DIAGONAL IDEMPOTENTS}

We begin with a canonical decomposition of idempotents into $2\times 2$ operator matrices.

\begin{lemma}
  \label{lem:idempotent-decomp}
  Let $D^2=D\in B(H)$ be an idempotent. 
  Then with respect to the decomposition $H=\ker^\perp D\oplus \ker D$, $D$ has the following block matrix form:
  \begin{equation*} 
  D = 
  \begin{pmatrix}
    I & 0 \\
    T & 0 \\
  \end{pmatrix},
  \end{equation*}
  where $I \in B(\ker^\perp D)$ is the identity operator and $T\in B(\ker^\perp D,\ker D)$ is a bounded operator which we call the \textbf{nilpotent part} of the idempotent $D$, short for the corner of the nilpotent operator 
  \(
  \begin{psmallmatrix}
    0 & 0 \\
    T & 0 \\
  \end{psmallmatrix}
  \). 
\end{lemma}

\noindent Note that the term `nilpotent part' is a natural slight abuse of language in that $T$ itself is not nilpotent; $T^2$ is not even defined.

\begin{proof}
  The only non-obvious fact we must prove is that the upper left-hand corner of $D$ is the identity on the compression to $\ker^\perp D$.
 To verify this let $x \oplus 0 \in \ker^\perp D$ be arbitrary and let $D(x \oplus 0) = y \oplus z$.
 Then because $D$ is idempotent one has
\begin{equation*} D(x \oplus 0) = D^2(x \oplus 0) = D(y \oplus z) = D(y \oplus 0). \end{equation*}
Since $x\oplus 0, y\oplus 0\in \ker^\perp D$ on which $D$ acts one-to-one, $x=y$. 
\end{proof}

An important stepping stone to our first main theorem is the following proposition in which the idempotent acts on $H\oplus H$ and its nilpotent part is normal.

\begin{proposition}
  \label{prop:stepping-stone}
  Suppose $H$ is separable infinite dimensional and the idempotent $D\in B(H \oplus H)$ has the respective block matrix form 
\begin{equation*} 
D =
\begin{pmatrix}
  I & 0 \\
  T & 0 \\
\end{pmatrix}
\end{equation*}
where $T\in B(H)$ is normal.
Then
\begin{enumerate}
\item \label{item:1} both $(\Im\, D)_\pm = (\Re\, \mathrm{e}^{\nicefrac{\mathrm{i}\pi}{2}}D)_\pm \in \mathcal{C}_1$ if and only if $T \in \mathcal{C}_1$;
\item \label{item:2} $(\Re\, \mathrm{e}^{\mathrm{i}\theta}D)_+ \notin \mathcal{C}_1$ for $\nicefrac{-\pi}{2}<\theta<\nicefrac{\pi}{2}$;
\item \label{item:3} $(\Re\, \mathrm{e}^{\mathrm{i}\theta}D)_- \in \mathcal{C}_1$ for $\nicefrac{-\pi}{2}<\theta<\nicefrac{\pi}{2}$ if and only if $T \in \mathcal{C}_2$.
\end{enumerate}
\end{proposition}

\begin{proof}
  The core of the proof is an analysis of the $2\times 2$ case followed by a straightforward application of the Borel functional calculus to the operator case.
  
  For $z\in \mathbb{C}$, let $A_z\in M_2(\mathbb{C})$ be given by
  \begin{equation*} 
  A_z \coloneqq{}
  \begin{pmatrix}
    1            & 0             \\
    z            & 0             \\
  \end{pmatrix}.
  \end{equation*}
  Then fixing $-\nicefrac{\pi}{2}<\theta\le\nicefrac{\pi}{2}$,
  \begin{equation*} 2(\Re\, \mathrm{e}^{\mathrm{i}\theta}A_z) = \mathrm{e}^{\mathrm{i}\theta}A_z + \mathrm{e}^{-\mathrm{i}\theta}A_z^* = 
  \begin{pmatrix}
    2\cos\theta  & \mathrm{e}^{-\mathrm{i}\theta}\bar z \\
    \mathrm{e}^{\mathrm{i}\theta}z & 0             \\
  \end{pmatrix}, 
  \end{equation*}
  which has characteristic polynomial $\det(\lambda-2(\Re\, \mathrm{e}^{\mathrm{i}\theta}A_z))=\lambda^2-2\cos\theta\lambda -\sabs{z}^2$. 
Hence the selfadjoint matrix $2(\Re\, \mathrm{e}^{\mathrm{i}\theta}A_z)$ has eigenvalues which depend on $z$ by 
\begin{equation}
  \label{eq:eigenvalues}
  \lambda_\pm(z) = \cos\theta\pm\sqrt{\cos^2\theta+\sabs{z}^2}.
  \end{equation}

  When $z\not=0$, normalized eigenvectors corresponding to these eigenvalues are
  \begin{equation}
    \label{eq:eigenvectors}
    x_+(z) =
    \begin{pmatrix}
      \frac{\lambda_+(z)}{\sqrt{\lambda_+^2(z)+\sabs{z}^2}} \\
      \frac{\mathrm{e}^{\mathrm{i}\theta}z}{\sqrt{\lambda_+^2(z)+\sabs{z}^2}} \\
    \end{pmatrix}
    \quad \text{and} \quad 
    x_-(z) =
    \begin{pmatrix}
      \frac{\lambda_-(z)}{\sqrt{\lambda_-^2(z)+\sabs{z}^2}} \\
      \frac{\mathrm{e}^{\mathrm{i}\theta}z}{\sqrt{\lambda_-^2(z)+\sabs{z}^2}} \\
    \end{pmatrix}
    .
  \end{equation}
  On the other hand, when $z=0$, the normalized eigenvectors are just the standard basis $x_+(0) =
  \begin{psmallmatrix}
    1 \\ 
    0 \\
  \end{psmallmatrix}
  $ and 
  $x_-(0) =
  \begin{psmallmatrix}
    0 \\ 
    1 \\
  \end{psmallmatrix}
  $. 
  
  We now return to the operator case. 
  Since $T\in B(H)$ is normal, the Borel functional calculus provides a $*$-homomorphism $\Phi:\mathcal{B}(\spec(T))\to W^*(T)$ from the bounded Borel functions on the spectrum of $T$ to the abelian von Neumann algebra generated by $T$ for which $\cramped{\id_{\spec T} \xmapsto{\Phi} T}$, where the identity function on $\spec T$ is $\cramped{\id_{\spec T}}(z) = z$ \cite[Theorem 5.2.9]{Kadison.Ringrose-1997a}.
  Moreover, since $\Phi$ is a $*$-homomorphism, it preserves the partial order on selfadjoint elements.
  Let $\mathbf{1}\in\mathcal{B}(\spec(T))$ denote the identity element (the map $z\mapsto 1$) of the algebra $\mathcal{B}(\spec(T))$, and $x_\pm^i$ ($i=1,2$) the coordinate functions of the eigenvectors obtained in (\ref{eq:eigenvectors}), which are bounded Borel functions on $\mathbb{C}$.
 Define 
  \begin{equation*} U \coloneqq{}
  \begin{pmatrix}
    \Phi(x_+^1) & \Phi(x_-^1) \\
    \Phi(x_+^2) & \Phi(x_-^2) \\
  \end{pmatrix}, 
  \end{equation*}
  which is unitary on $H\oplus H$ because $\Phi$ is a $*$-homomorphism and $\{x_\pm(z)\}$ form a basis for $\mathbb{C}^2$ for every $z\in\mathbb{C}$. That is, because the $z$-functions $x_+^1x_-^1+x_+^2x_-^2 \equiv 0$ and $\abs{x_{\pm}^1}^{2}+\abs{x_{\pm}^2}^{2} \equiv 1$ and $\Phi(\mathbf{1})=I$ is the identity on $H$. And for what follows, recall $\Phi(\id_{\spec T}) = T$.

  Furthermore, because
  \begin{equation*}
  D =
  \begin{pmatrix}
    \Phi(\mathbf{1}) & \Phi(0\cdot\mathbf{1}) \\
    \Phi(\id_{\spec T}) & \Phi(0\cdot\mathbf{1}) \\
  \end{pmatrix},
  \end{equation*}
  where here $\cdot$ denotes multiplication by scalars in the algebra $\mathcal{B}(\spec(T))$ and hence $0\cdot\mathbf{1}$ is simply the zero function, and so also
  \begin{equation*}
  2(\Re\, \mathrm{e}^{\mathrm{i}\theta}D) =
  \begin{pmatrix}
    \Phi(2\cos\theta\cdot\mathbf{1}) & \Phi(\mathrm{e}^{-\mathrm{i}\theta}\cdot\overline{\id_{\spec T}}) \\
    \Phi(\mathrm{e}^{\mathrm{i}\theta}\cdot\id_{\spec T}) & \Phi(0\cdot\mathbf{1}) \\
  \end{pmatrix},
  \end{equation*}
  one obtains
  \begin{equation*}
  U^*2(\Re\, \mathrm{e}^{\mathrm{i}\theta}D)U =
  \begin{pmatrix}
    \Phi(\lambda_+) & \Phi(0\cdot\mathbf{1}) \\
    \Phi(0\cdot\mathbf{1}) & \Phi(\lambda_-) \\
  \end{pmatrix}.
  \end{equation*}
  When $-\nicefrac{\pi}{2}<\theta<\nicefrac{\pi}{2}$ one has $\cos\theta>0$, and therefore $\lambda_+\ge 0$ and $\lambda_-\le 0$. 
  Hence
  \begin{equation}
    \label{eq:1}
    (U^*2(\Re\, \mathrm{e}^{\mathrm{i}\theta}D)U)_+ = \Phi(\lambda_+)\oplus 0 \quad\text{and}\quad (U^*2(\Re\, \mathrm{e}^{\mathrm{i}\theta}D)U)_- = 0\oplus\Phi(-\lambda_-).
  \end{equation}
  Moreover, for all $z\in\mathbb{C}$, 
  \begin{equation*}
    \lambda_+(z) = \cos\theta + \sqrt{\cos^2\theta+\sabs{z}^2} \ge 2\cos\theta.
  \end{equation*}
  Furthermore, for the same range of $\theta$, and for all $z$ lying inside the closed ball $\bar{B}(0;\snorm{T})\supseteq \spec T$, 
  \begin{align*}
    -\lambda_-(z) = \sqrt{\cos^2\theta+\sabs{z}^2} - \cos\theta & = \frac{\abs{z}^2}{\cos\theta + \sqrt{\cos^2\theta+\sabs{z}^2}} \\ & \ge \frac{\abs{z}^2}{\cos\theta + \sqrt{\cos^2\theta+\snorm{T}^2}},
  \end{align*}
  and
  \begin{equation*}
    -\lambda_-(z) = \frac{\abs{z}^2}{\cos\theta + \sqrt{\cos^2\theta+\sabs{z}^2}} \le \frac{\sabs{z}^2}{2\cos\theta}.
  \end{equation*}
  From these inequalities, as Borel functions on the spectrum of $T$, we have the following $z$-function inequalities for $-\nicefrac{\pi}{2}<\theta<\nicefrac{\pi}{2}$:
  \begin{equation}
    \label{eq:10}
    \lambda_+ \ge 2\cos\theta\cdot\mathbf{1} 
    \quad\text{and}\quad
    C_1\cdot\abs{\id_{\spec T}}^2 \le -\lambda_- \le C_2\cdot\abs{\id_{\spec T}}^2,
  \end{equation}
  where $C_1,C_2$ are the positive constants given by $C_1 \coloneqq{} \frac{1}{\cos\theta + \sqrt{\cos^2\theta+\snorm{T}^2}}$ and $C_2 \coloneqq{} \frac{1}{2\cos\theta}$.
 After applying $\Phi$ to these inequalities, one has $\Phi(\lambda_+) \ge (2\cos\theta) I$ and $C_1\abs{T}^2 \le \Phi(-\lambda_-) \le C_2\abs{T}^2$.
 Applying the trace yields
  \begin{equation}
    \label{eq:11}
    \begin{aligned}
      \trace_{H\oplus H}(U^*2(\Re\, \mathrm{e}^{\mathrm{i}\theta}D)U)_+ &= \trace_{H\oplus H}(\Phi(\lambda_+)\oplus 0) \\
      &= \trace_H \Phi(\lambda_+) \ge (2\cos\theta) \trace_H  I = \infty,
    \end{aligned}
  \end{equation}
  and
  \begin{equation}
    \label{eq:13}
    C_1 \trace_H\abs{T}^2 \le \trace_H \Phi(-\lambda_-) \le C_2 \trace_H\abs{T}^2.
  \end{equation}
  Because
  \begin{equation}
    \label{eq:12}
    \trace_{H\oplus H}(U^*2(\Re\, \mathrm{e}^{\mathrm{i}\theta}D)U)_- = \trace_{H\oplus H}(0\oplus\Phi(-\lambda_-)) = \trace_H \Phi(-\lambda_-)
  \end{equation}
  and
  \begin{equation}
    \label{eq:14}
    \trace_{H\oplus H}(U^*2(\Re\, \mathrm{e}^{\mathrm{i}\theta}D)U)_\pm = \trace_{H\oplus H}(2(\Re\, \mathrm{e}^{\mathrm{i}\theta}D))_\pm, 
  \end{equation}
  inequalities (\ref{eq:11})--(\ref{eq:14}) prove \ref{item:2} and \ref{item:3}.
 To prove \ref{item:1}, simply notice that when $\theta=\nicefrac{\pi}{2}$, we have $\lambda_+ = -\lambda_- = \sabs{\id_{\spec T}}$ and apply the same arguments as above in \eqref{eq:11} and \eqref{eq:12} along with the fact that $\Phi(\sabs{\id_{\spec T}}) = \abs{T}$
\end{proof}

The following remark shows that idempotents can be decomposed even further than the $2\times 2$ matrix of Lemma~\ref{lem:idempotent-decomp}.
  
\begin{remark}
  \label{rem:finer-idempotent-decomp}
  With the same notation as Lemma~\ref{lem:idempotent-decomp}, we may further decompose the underlying space as $\ker^{\perp}D=\ker T\oplus \ker^{\perp}T$ and $\ker D = \range^{\perp}T\oplus \overline{\range T}$, where $\ker^{\perp}T\coloneqq{} \ker^{\perp}D\ominus \ker T$ and $\range^{\perp}T\coloneqq{}\ker D\ominus \overline{\range T}$. 
  With respect to the ordering of subspaces $H = \ker T \oplus \range^\perp T \oplus \ker^\perp T \oplus \overline{\range T}$ one can write
  \begin{equation*}
    D =
    \begin{pmatrix}
      I & 0 & 0               & 0 \\
      0 & 0 & 0               & 0 \\
      0 & 0 & I               & 0 \\
      0 & 0 & \tilde T        & 0 \\
    \end{pmatrix},
  \end{equation*}
  where $\tilde T \in B(\ker^\perp T,\overline{\range T})$, and the identity operators act on the appropriate spaces. 
  In the decomposition above we have used the ordering of subspaces $\ker T \oplus \range^\perp T \oplus \ker^\perp T \oplus \overline{\range T}$, which makes it clear that $D$ can be written as the direct sum of a projection and another idempotent.
  It is possible for this decomposition to degenerate into simpler ones if, say, $\ker T = \{0\}$, in which the first row and column would disappear.
  Other rows and columns would disappear if their corresponding subspaces were zero, but none of this is problematic. 
  
  If $Q_3:\ker^{\perp} T\to H$ denotes the (linear) inclusion operator and $Q_4: H\to\overline{\range T}$ the projection operator, then $\tilde T = Q_4TQ_3$. 
  From this it is clear that $\tilde T$ is injective and has dense range.
  Furthermore, if $\tilde T =
  U\sabs{\tilde T}$ is the polar decomposition for $\tilde T$, then $U:\ker^\perp T \to \overline{\range T}$ is unitary (i.e., a surjective isometry, see \cite[Problem 134 and corollaries]{Halmos1982}). 
  Conjugating $D$ by the unitary $V \coloneqq{} I \oplus I \oplus I \oplus U^* \in B(H,H')$, where $H' = \ker T \oplus
  \range^\perp T \oplus \ker^\perp T \oplus \ker^\perp T$, one obtains
  \begin{equation*}
  D' \coloneqq{} VDV^* = 
  \begin{pmatrix}
    I & 0 & 0               & 0 \\
    0 & 0 & 0               & 0 \\
    0 & 0 & I               & 0 \\
    0 & 0 & \sabs{\tilde T} & 0 \\
  \end{pmatrix}.
  \end{equation*}
\end{remark}

We need one more lemma before we can prove our main theorem for this section.

\begin{lemma}
  \label{lem:positive-ideal-perturb}
  Let $\mathcal{I}$ be a two-sided ideal of $B(H)$ and let $B=B^*\in\mathcal{I}$ and $A=A^*\in B(H)$.
 Then $A_+\in\mathcal{I}$ if and only if $(A+B)_+\in\mathcal{I}$. Similarly, $A_-\in\mathcal{I}$ if and only if $(A+B)_-\in\mathcal{I}$
\end{lemma}

\begin{proof}
  Let $R_{A_+}$ be the range projection of the positive part $A_+$ of $A$.
 Then since $A+B \le (A+B)_+$, one has $A \le (A+B)_+ - B$.
 Therefore
  \begin{equation*}
  A_+ = R_{A_+}AR_{A_+} \le R_{A_+}((A+B)_+ - B)R_{A_+},
  \end{equation*}
  and hence $A_+\in\mathcal{I}$ whenever $(A+B)_+\in\mathcal{I}$. 
  Here we are using the fact that two-sided ideals of $B(H)$ are hereditary, which is a well-known consequence of Calkin's characterization of ideals of $B(H)$ in terms of their $s$-numbers in \cite{Cal-1941-AoM2}.

  For the other implication, make the substitutions $A \mapsto A+B$, $B \mapsto - B$ and apply the result just proved.
 More precisely, one obtains 
  \begin{equation*} 
  (A+B)_+ \le P((A+B)-B)_+ + B)P = P(A_+ + B)P,
  \end{equation*}
  where $P \coloneqq{} R_{(A+B)_+}$.
  Hence $(A+B)_+\in\mathcal{I}$ if $A_+\in\mathcal{I}$. 
  
  To see that $A_-\in\mathcal{I}$ if and only if $(A+B)_-\in\mathcal{I}$, note that $A_- = (-A)_+$ and apply the result just proved.
\end{proof}

We are now in a position to prove our first main theorem.

\firstmaintheorem

\begin{proof}
  The implication \ref{item:mthm1}$\implies$\ref{item:mthm2} is clear, as are the implications \ref{item:mthm4}$\implies$\ref{item:mthm5}$\implies$\ref{item:mthm6}. 
The implication \ref{item:mthm3}$\implies$\ref{item:mthm4} is a direct consequence of Theorem~\ref{thm:subsequence-zero-diagonalizability}, for if $R\{\trace T\}=\mathbb{C}$, then there exists a basis $\mathfrak{e}$ with respect to which $\trace_\mathfrak{e} T = 0$, and thus by Theorem~\ref{thm:subsequence-zero-diagonalizability}, $T$ is zero-diagonal. 
Hence the main thrust of this theorem is proving the implications \ref{item:mthm6}$\implies$\ref{item:mthm2}$\implies$\ref{item:mthm3} and \ref{item:mthm2}$\implies$\ref{item:mthm1}.

  For the remainder of the proof we use Lemma~\ref{lem:positive-ideal-perturb}, Remark~\ref{rem:finer-idempotent-decomp}, Proposition~\ref{prop:stepping-stone}, and Theorem~\ref{thm:trace-range-shape} to prove the implications \ref{item:mthm6}$\implies$\ref{item:mthm2}$\implies$\ref{item:mthm3} which, with the above paragraph, establishes the equivalences \ref{item:mthm2}--\ref{item:mthm6}.
  Having demonstrated these equivalences, we prove \ref{item:mthm4}$\implies$\ref{item:mthm1} in lieu of  \ref{item:mthm2}$\implies$\ref{item:mthm1}.

  \begin{proof}[\ref{item:mthm6}$\implies$\ref{item:mthm2}] \renewcommand{\qedsymbol}{} We prove the contrapositive, that the nilpotent part of $D$ is Hilbert--Schmidt implies $R\{\trace D\}=\emptyset$.
  Suppose the nilpotent part of $D$ is Hilbert--Schmidt.

    \emph{Case 1: The nilpotent part of $D$ has finite rank.} 
    
    \noindent By Lemma~\ref{lem:idempotent-decomp}, $D$ has the form
    \begin{equation*}
    D =
    \begin{pmatrix}
      I & 0 \\
      T & 0 \\
    \end{pmatrix}.
    \end{equation*}
    Set
    \begin{equation*}
    A =
    \begin{pmatrix}
      I & 0 \\
      0 & 0 \\
    \end{pmatrix}
    \quad\text{and}\quad
    B =
    \frac{1}{2}
    \begin{pmatrix}
      0 & T^* \\
      T & 0  \\
    \end{pmatrix},
    \end{equation*}
    and so $\Re D=A+B$. 
    By hypothesis, $T$ has finite rank hence $B$ has finite rank.
 Since $A=A_+\notin\mathcal{C}_1$ and $B\in\mathcal{C}_1$ because $B$ has finite rank, $(\Re\, D)_+=(A+B)_+\notin\mathcal{C}_1$ by Lemma~\ref{lem:positive-ideal-perturb}.
 However, $A_-=0\in\mathcal{C}_1$ and so again Lemma~\ref{lem:positive-ideal-perturb} ensures $(\Re\, D)_-=(A+B)_-\in\mathcal{C}_1$. 
Therefore,
 \begin{equation}
   \label{eq:theta-zero-1}
   (\Re\, D)_+\notin\mathcal{C}_1 \quad\text{and}\quad (\Re\, D)_-\in\mathcal{C}_1.
 \end{equation}
 Then Theorem~\ref{thm:trace-range-shape}(iv) with $\theta=0$ ensures $R\{\trace D\}=\emptyset$.

    \emph{Case 2: The nilpotent part of $D$ has infinite rank.}

    \noindent By Remark~\ref{rem:finer-idempotent-decomp}, write 
    \begin{equation*} 
    D' = 
    \begin{pmatrix}
      I & 0 & 0               & 0 \\
      0 & 0 & 0               & 0 \\
      0 & 0 & I               & 0 \\
      0 & 0 & \sabs{\tilde T} & 0 \\
    \end{pmatrix}, 
    \end{equation*}
    and from $\tilde T = Q_4TQ_3$ we know that $\sabs{\tilde T}$ is Hilbert--Schmidt, and since $\tilde T$ has dense range in $\overline{\range T}$ which is infinite dimensional $\tilde T$, and hence also $\sabs{\tilde T}$, have infinite rank.
 Define $J:=\ker T\oplus \range^\perp T$ and $K:=\ker^\perp T$, then set $P\in B(J)$ and $\tilde D\in B(K\oplus K)$ to 
    \begin{equation*}
    P \coloneqq{}
    \begin{pmatrix}
      I & 0 \\
      0 & 0 \\
    \end{pmatrix}
    \quad\text{and}\quad
    \tilde D \coloneqq{}
    \begin{pmatrix}
      I & 0 \\
      \sabs{\tilde T} & 0 \\
    \end{pmatrix}
    \end{equation*}
    Then $\tilde D$ satisfies the conditions of Proposition~\ref{prop:stepping-stone} and so $(\Re\, D')_+ = P \oplus (\Re\, \tilde D)_+ \notin\mathcal{C}_1$ because 
    \begin{equation*}
    \trace_H(P \oplus (\Re\, \tilde D)_+) = \trace_J P + \trace_{K\oplus K}(\Re\, \tilde D)_+  \ge \trace_{K\oplus K}(\Re\, \tilde D)_+ \underset{\ref{prop:stepping-stone}\ref{item:mthm2}}{=} \infty.
    \end{equation*}
    Furthermore, $(\Re\, D')_- \in \mathcal{C}_1$ because
    \begin{equation*} 
    (\Re\, D')_- = 0 \oplus (\Re\, \tilde D)_-
    \end{equation*}
    and $(\Re\, \tilde D)_-\in\mathcal{C}_1$ by Proposition~\ref{prop:stepping-stone}(iii) since the nilpotent part $\sabs{\tilde T}$ of $\tilde D$ is Hilbert--Schmidt.
    Therefore $(\Re\, D')_+\notin \mathcal{C}_1$ and $(\Re\, D')_-\in \mathcal{C}_1$, and also via unitary equivalence
    \begin{equation}
      \label{eq:theta-zero-2}
      (\Re\, D)_+\notin \mathcal{C}_1 \quad\text{and}\quad (\Re\, D)_-\in \mathcal{C}_1.
    \end{equation}
    Thus by Theorem~\ref{thm:trace-range-shape}(iv), one has that $R\{\trace D\} = \emptyset$.
\end{proof}
\begin{proof}[\ref{item:mthm2}$\implies$\ref{item:mthm3}] \renewcommand{\qedsymbol}{}
  Suppose the nilpotent part of $D$ is not Hilbert--Schmidt. 
  Then just like in Case 2 above use Remark~\ref{rem:finer-idempotent-decomp} to decompose $D' = P \oplus \tilde D$, with $\tilde D$ satisfying the conditions of Proposition~\ref{prop:stepping-stone}(ii). 
  Then for $-\nicefrac{\pi}{2}<\theta<\nicefrac{\pi}{2}$
  \begin{equation*} 
  \trace_H(\Re\, \mathrm{e}^{\mathrm{i}\theta}D')_+ = \trace_J(\cos\theta P) + \trace_{K\oplus K}(\Re\, \mathrm{e}^{\mathrm{i}\theta}\tilde D)_+ \ge \trace_{K\oplus K}(\Re\, \mathrm{e}^{\mathrm{i}\theta}\tilde D)_+ = \infty. 
  \end{equation*}
  Furthermore, since the nilpotent part $T$ of $D$ is not Hilbert--Schmidt, and hence $\sabs{\tilde T}$ is not Hilbert--Schmidt, one has
  \begin{equation*} 
  \trace_{H}(\Re\, \mathrm{e}^{\mathrm{i}\theta}D')_- = 0 + \trace_{K\oplus K}(\Re\, \mathrm{e}^{\mathrm{i}\theta}\tilde D)_- = \infty
  \end{equation*}
  by Proposition~\ref{prop:stepping-stone}(iii).
  Finally, since $\sabs{\tilde T}$ is not Hilbert--Schmidt, neither is it trace-class. 
  Therefore, by Proposition~\ref{prop:stepping-stone}(i) 
  \begin{equation*}
  \trace_H(\Re\, \mathrm{e}^{\nicefrac{\mathrm{i}\pi}{2}}D')_{\pm} = \trace_H(\Im\, D')_\pm = 0 + \trace_{K\oplus K}(\Im\, \tilde D)_\pm = \infty.
  \end{equation*}
  Thus we have proven that $\trace(\Re\, \mathrm{e}^{\mathrm{i}\theta}D)_\pm=\trace(\Re\, \mathrm{e}^{\mathrm{i}\theta}D')_\pm=\infty$ for all $-\nicefrac{\pi}{2}<\theta\le\nicefrac{\pi}{2}$ and hence also for all $\theta\in\mathbb{R}$, and so by Theorem~\ref{thm:trace-range-shape}(iv) one has $R\{\trace D\}=\mathbb{C}$. 
\end{proof}
Having established the equivalence of \ref{item:mthm2}--\ref{item:mthm6} and the implication \ref{item:mthm1}$\implies$\ref{item:mthm2}, it suffices to prove \ref{item:mthm4}$\implies$\ref{item:mthm1}. 
We will in fact prove the contrapositive. 
To this end, suppose $D$ is a Hilbert--Schmidt perturbation of a projection. 
That is, $D=P+K$ where $P$ is a projection and $K\in\mathcal{C}_2$.
 Because $D$ is idempotent one has
\begin{equation*}
P+K = D = D^2 = P^2 + PK + KP + K^2 = P + PK + KP + K^2, 
\end{equation*}
and so
\begin{equation}
  \label{eq:K=PK+KP+K^2}
  K = PK + KP + K^2 \qquad\text{and}\qquad PKP = 2PKP + PK^2P,
\end{equation}
so $PKP=-PK^2P\in\mathcal{C}_1$. 
Similarly for $P^{\perp}$ one has $P^\perp KP^\perp = P^\perp K^2P^\perp\in\mathcal{C}_1$.
Therefore, with respect to the decomposition $H = PH \oplus P^\perp H$, one has
\begin{equation*}
K =
\begin{pmatrix}
K_1 & K_2 \\
K_3 & K_4 \\
\end{pmatrix},
\end{equation*}
where $K_1,K_4\in\mathcal{C}_1$ and $K_2,K_3\in\mathcal{C}_2$. 
A technical note is that $P$ must have infinite rank.
Otherwise, if $P$ were finite rank, then so also $K_2,K_3$ would be finite rank. 
Hence $K$ would be trace-class, and so also would $D=P+K$, which contradicts the fact that $D$ is an infinite rank idempotent because of Lemma~\ref{lem:idempotent-decomp}.
Thus relative to $H = PH \oplus P^\perp H$ we may write
\begin{equation*}
D =
\overbrace{
  \begin{pmatrix}
    I & K_2 \\
    K_3 & 0 \\
  \end{pmatrix}
}^{D_1}
+
\overbrace{
  \begin{pmatrix}
    K_1 & 0 \\
    0 & K_4 \\
  \end{pmatrix}
}^{D_2}.
\end{equation*}
Moreover, because 
\begin{equation*}
\Re\, D_1 = \Re
\overbrace{
\begin{pmatrix}
  I & 0 \\
  K_2^*+K_3 & 0 \\
\end{pmatrix}
}^{\tilde D_1}
\end{equation*}
and $K_2^*+K_3\in\mathcal{C}_2$, by the proof of \ref{item:mthm6}$\implies$\ref{item:mthm2} (see \eqref{eq:theta-zero-1} and \eqref{eq:theta-zero-2} for Cases 1 and 2),  $(\Re\, D_1)_+=(\Re\, \tilde D_1)_+\notin\mathcal{C}_1$ but $(\Re\, D_1)_-=(\Re\, \tilde D_1)_-\in\mathcal{C}_1$. 
So by Theorem~\ref{thm:trace-range-shape}(iv), $R\{\trace D_1\}=\emptyset$ and hence $D_1$ does not have an absolutely summable diagonal in \emph{any} basis.
 Because $D_2\in\mathcal{C}_1$, its diagonal in \emph{any} basis is absolutely summable.
 Therefore, there is no basis in which $D=D_1+D_2$ has a zero diagonal, which completes the proof.
\end{proof}

The following corollary answers Question~\ref{que:zero-diagonal-finite-rank} due to Jasper.

\begin{corollary}
  \label{cor:answer-q2}
  A nonzero idempotent $D$ is zero-diagonal if and only if it is not a Hilbert--Schmidt perturbation of a projection.
\end{corollary}

\begin{proof}
  If $D$ has infinite rank, this is handled by Theorem~\ref{thm:idempotent-zero-diagonal}.
If $D$ has finite rank, then so does the nilpotent part of $D$.
Thus $D$ is a finite rank (and hence Hilbert--Schmidt) perturbation of the zero projection.
Furthermore, by Lemma~\ref{lem:idempotent-decomp} $\trace D = \rank D >0$ for finite rank idempotents, and so $D$ is not zero-diagonal. 
\end{proof}

In the case of infinite rank projections with infinite dimensional kernel, the next corollary is a strengthening of the result due to Fan \cite[Theorem 3]{Fan-1984-TotAMS} that an operator $T$ is a norm limit of zero-diagonal operators if and only if $0\in W_e(T)$, the essential numerical range. 
For $P$ a projection, $0\in W_e(P)$ if and only if $\trace P^\perp = \infty$, and thus Fan's result guarantees such projections are a norm limit of zero-diagonal operators. 
However, we take this a step further by proving these zero-diagonal operators may be taken to be idempotent so long as $\trace P = \infty$ as well. 

\begin{corollary}
  Every projection $P$ with $\trace P =\trace P^\perp =\infty$ is a norm limit of zero-diagonal idempotents. 
\end{corollary}

\begin{proof}
  For $P=I\oplus 0$ consider idempotents 
  \( 
  \begin{psmallmatrix}
    I & 0 \\
    T & 0 \\
  \end{psmallmatrix}
  \)
  whose nilpotent part has arbitrarily small norm but is not Hilbert--Schmidt and apply Theorem 2.5 (ii)$\iff$(iv).
\end{proof}

\subsection*{Constructing bases to achieve zero-diagonality}

The proof of Theorem~\ref{thm:idempotent-zero-diagonal} was existential in the sense that it did not explicitly construct a basis in which a given idempotent has zero diagonal.
The remainder of this section is devoted to providing an algorithm for constructing such a basis when it exists (i.e., when the idempotent is not a Hilbert--Schmidt perturbation of a projection, which is included in the case when $\dim\ker D = \infty = \dim\ker^{\perp} D$).
As with the proof of Proposition~\ref{prop:stepping-stone}, a careful consideration first of the $2\times 2$ case is in order.

\begin{remark}
  \label{rem:2x2-theta-minimum}
  Consider a $2\times 2$ idempotent matrix, $D$, and the counterclockwise rotation matrix through an angle $\theta$, $R_\theta$, given by the formulas
  \begin{equation*} 
  D = 
  \begin{pmatrix}
    1 & 0 \\
    d & 0 \\
  \end{pmatrix}
  \quad \text{and}\quad
  R_\theta = 
  \begin{pmatrix}
    \cos\theta & -\sin\theta \\
    \sin\theta & \cos\theta \\
  \end{pmatrix},
  \end{equation*}
  where $d\ge 0$. 
  Conjugating $D$ by $R_\theta$ is equivalent to changing the basis for $\mathbb{C}^2$:
  \begin{align*}
    R_{-\theta}DR_\theta &= 
  \begin{pmatrix}
    \cos\theta & \sin\theta \\
    -\sin\theta & \cos\theta \\
  \end{pmatrix}
  \begin{pmatrix}
    1 & 0 \\
    d & 0 \\
  \end{pmatrix}
  \begin{pmatrix}
    \cos\theta & -\sin\theta \\
    \sin\theta & \cos\theta \\
  \end{pmatrix}
  \\
  &= 
  \begin{pmatrix}
    \cos^2\theta+d\sin\theta\cos\theta & -\sin\theta\cos\theta-d\sin^2\theta \\
    -\sin\theta\cos\theta+d\cos^2\theta & \sin^2\theta-d\sin\theta\cos\theta \\
  \end{pmatrix}
  \\ 
  &= 
  \begin{pmatrix}
    \frac{1+\cos 2\theta+d\sin 2\theta}{2} & -\sin\theta\cos\theta-d\sin^2\theta \\
    -\sin\theta\cos\theta+d\cos^2\theta & \frac{1-\cos 2\theta-d\sin 2\theta}{2} \\
  \end{pmatrix}.
  \end{align*}
  Elementary calculus shows that the minimum diagonal entry occurs when $\theta=\frac{\arctan d}{2}$ and corresponds to a negative value of 
  \begin{equation*}
  -d^- \coloneqq{} \frac{1}{2}\left(1-\sqrt{1+d^2}\right) = \frac{-d^2}{2\left(1+\sqrt{1+d^2}\right)}. 
  \end{equation*}
  Since the trace is basis independent, the other diagonal entry is necessarily $1+d^-$. 
  Furthermore, by continuity of the diagonal entries as a function of $\theta$, for any value $x$ with $-d^-\le x\le 0$, there is some $\theta$ for which one of the diagonal entries is $x$. 
\end{remark}

We require the following elementary result in linear algebra \cite[Page 77, Problem 3]{HJ-1990-MA}.
It's proof by induction is straightforward and we include it here for completeness.

\begin{lemma}
  \label{lem:zero-diagonal-trace-zero}
  Let $X\in M_n(\mathbb{C})$. 
Then $\trace X = 0$ if and only if there is a basis in which $X$ has zero diagonal.
\end{lemma}

\begin{proof}
  One direction is clear, so suppose $\trace X = 0$. 
  We proceed by induction on the size $n$ of the $n \times n$ matrix $X$. 
  The case $n=1$ is clear.
  Given any basis $\{e_j\}_{j=1}^n$, one has 
  \begin{equation*}
    0 = \trace X = \sum_{j=1}^n (Xe_j,e_j), 
  \end{equation*}
  and therefore also $0 = \sum_{j=1}^n \frac{(Xe_j,e_j)}{n}$. 
  Thus zero is in the convex hull of $\{(Xe_j,e_j)\} \subseteq W(X)$. 
  But the Toeplitz--Hausdorff Theorem, that the numerical range $W(X)$ is convex, ensures $0 \in W(X)$. 
  So there is some unit vector $f_1$ for which $(Xf_1,f_1) = 0$. 
  Let $P$ be the projection onto the orthogonal complement of $f_1$. 
  Then we find 
  \begin{equation*}
    0 = \trace X = (Xf_1,f_1) + \trace(PXP) = \trace(PXP). 
  \end{equation*}
  The matrix $PXP$ can be viewed as being of size $(n-1) \times (n-1)$ by expressing it in a basis which contains $f_1$ and deleting the row and column corresponding to $f_1$ (which consist solely of zeros). 
  By applying the inductive hypothesis to $PXP$ we obtain orthonormal vectors $f_2,\ldots,f_n$ which are orthogonal to $f_1$ and satisfy $(Xf_j,f_j) = (XPf_j,Pf_j) = (PXPf_j,f_j) = 0$ for $2 \le j \le n$. 
  Therefore $\{f_j\}_{j=1}^n$ is a basis with respect to which $X$ has zero diagonal.
\end{proof}

\noindent We will use the following obvious corollary of Lemma~\ref{lem:zero-diagonal-trace-zero} extensively in the next section.

\begin{corollary}
  \label{cor:trace-nlambda-diagonal-lambda}
  Let $X\in M_n(\mathbb{C})$. 
Then $\trace X = n\lambda$ if and only if there is a basis in which $X$ has constant diagonal sequence $\lambda$. 
More generally, if $X\in B(H)$ with basis $\{e_n\}_{n\in\mathbb{N}}$ and $\langle n_k \rangle_{k=1}^m$ a finite subsequence of $\mathbb{N}$ with restricted trace $\sum_{k=1}^m (Xe_{n_k},e_{n_k}) = m\lambda$, then there is an orthonormal set $\{f_{n_k}\}_{k=1}^m$ for which $(Xf_{n_k},f_{n_k})=\lambda$ for $k=1,\ldots,m$ and $\spans \{f_{n_k}\}_{k=1}^m = \spans \{e_{n_k}\}_{k=1}^m$. 
\end{corollary}

\begin{proof}
  For $X\in M_n(\mathbb{C})$ apply Lemma~\ref{lem:zero-diagonal-trace-zero} to $X-\lambda I$ and note that $\lambda I$ has constant diagonal sequence $\lambda$ with respect to \emph{any} basis. 

  For the general case $X\in B(H)$, let $P$ be the projection on $\spans \{e_{n_k}\}_{k=1}^m$ and apply the matrix result to $PXP$. 
Then simply notice that $(PXPf_{n_k},f_{n_k}) = (XPf_{n_k},Pf_{n_k}) = (Xf_{n_k},f_{n_k})$. 
\end{proof}

We are now ready to provide our algorithm. It requires an elementary theoretical first step with all succeeding steps algorithmic.

\begin{algorithm}
  Suppose that $D\in B(H)$ is not a Hilbert--Schmidt perturbation of a projection.
Then the following explicitly constructs (i.e., gives an algorithm for producing) a basis in which $D$ is zero-diagonal.
\end{algorithm}

\begin{proof}[Construction]
  By Theorem~\ref{thm:idempotent-zero-diagonal}, the nilpotent part of $D$ is not Hilbert--Schmidt.
  Then by introduction ((ii)-\eqref{eq:2}, contrapositive), there exists a basis in which the diagonal of the nilpotent part is not square-summable. 
  That is, there exists a basis for $H$ for which
  \begin{equation}
    \label{eq:nilpotent-part-diag-not-ell-2}
    D = 
    \left(
      \begin{array}{ccc|ccc}
        1      & \cdots & 0         & 0      & \cdots & 0      \\ 
        \vdots & 1      & \vdots    & \vdots & \ddots & \vdots \\
        0      & \cdots & \ddots    & 0      & \cdots & \ddots \\ \hline
        d_1    & \cdots & {\Huge *} & 0      & \cdots & 0      \\
        \vdots & d_2    & \vdots    & \vdots & \ddots & \vdots \\
        {\Huge *}      & \cdots & \ddots & 0      & \cdots & \ddots \\
      \end{array}
    \right)
  \end{equation}
  with $\langle d_n \rangle\in\ell^\infty\setminus\ell^2$. 
  Furthermore, by conjugating by a unitary $U$ of the form $U = I \oplus \diag\langle u_n \rangle$, we may even assume without loss of generality that $d_n\ge 0$.
  Let the basis which gives the form equation \eqref{eq:nilpotent-part-diag-not-ell-2} be $\mathfrak{e}\coloneqq{}\{e_n,e'_n\}_{n\in\mathbb{N}}$.
  We will transform these into a new basis $\mathfrak{f}\coloneqq{}\{f_n,f'_n\}_{n\in\mathbb{N}}$ for which $\spans\{e_n,e'_n\} = \spans\{f_n,f'_n\}$ for each $n\in\mathbb{N}$. 
  Specifically, $f_n = \cos\theta_ne_n+\sin\theta_ne'_n$ and $f'_n = -\sin\theta_ne_n+\cos\theta_ne'_n$ form a rotation of the pair $e_n,e'_n$ through an angle $\theta_n$ which we will choose momentarily.
  
  Recall \autoref{rem:2x2-theta-minimum} and notice that
  \begin{equation*}
  \sum_{n=1}^\infty d^-_n = \sum_{n=1}^\infty \frac{d_n^2}{2\big(1+\sqrt{1+d_n^2}\big)} \ge \frac{1}{2\big(1+\sqrt{1+\snorm{\langle d_n \rangle}_\infty^2}\big)} \sum_{n=1}^\infty d_n^2 = \infty.
  \end{equation*}
  Let $m_1$ be the smallest integer for which $\sum_{n=1}^{m_1} d^-_n \ge 1+d^-_1$.
Necessarily $m_1 \ge 2$. 
  Now define $\theta_n = \frac{\arctan d_n}{2}$ for $1\le n< m_1$, hence by \autoref{rem:2x2-theta-minimum}, $-d_n^-=(Df'_n,f'_n)$ and $1+d_n^-=(Df_n,f_n)$. 
  Our choice of $m_1$ guarantees
  \begin{equation*}
    \sum_{n=1}^{m_1-1} d^-_n < 1+d^-_1 \le \sum_{n=1}^{m_1} d^-_n,
    \quad\text{and thus}\quad
    -d^-_{m_1} \le -1-d^-_1 + \sum_{n=1}^{m_1-1} d^-_n < 0.
  \end{equation*}
  For the latter, using the continuity described in \autoref{rem:2x2-theta-minimum}  (last sentence) choose $\theta_{m_1}$ so that 
  \begin{equation*}
    (Df'_{m_1},f'_{m_1}) = -1-d^-_1 + \sum_{n=1}^{m_1-1} d^-_n,
  \end{equation*}
  and therefore
  \begin{equation*} \sum_{n=1}^{m_1} -(Df'_n,f'_n) = \sum_{n=1}^{m_1-1} d^-_n - (Df'_{m_1},f'_{m_1}) = 1+d^-_1 = (Df_1,f_1). \end{equation*}
  We will now inductively define the sequences $\langle m_k \rangle$ and $\langle \theta_n \rangle$ in the following interwoven fashion. 
  Suppose that these sequences are already defined up to $m_{k-1}$ and $\theta_{m_{k-1}}$. 
  Let $m_k$ be the smallest positive integer for which 
  \begin{equation*}
    \sum_{n=m_{k-1}+1}^{m_k} d^-_n \ge 1 - (Df'_k,f'_k) = (Df_k,f_k). 
  \end{equation*}
  Then for $m_{k-1} < n < m_k$, let $\theta_n =  \frac{\arctan d_n}{2}$, and as above let $\theta_{m_k}$ be chosen so as to satisfy
  \begin{equation*}
    \sum_{n=m_{k-1}+1}^{m_k} -(Df'_n,f'_n) = \sum_{n=m_{k-1}+1}^{m_k-1} d^-_n - (Df'_{m_k},f'_{m_k}) = 1 - (Df'_k,f'_k) = (Df_k,f_k).
  \end{equation*}
  
  Finally observe from this that with respect to the basis $\{f_n,f'_n\}_{n\in\mathbb{N}}$ the diagonal sequence of $D$ can be partitioned into finite subsets $\{A_k\}_{k\in\mathbb{N}}$ for which the sum over each subset is zero. 
  Indeed, let $A_k$ consist of the diagonal entries corresponding to the basis elements $\mathfrak{f}_k \coloneqq{}\{f_k,f'_{m_{k-1}+1},\ldots,f'_{m_k}\}$. 
  So for each $k\in\mathbb{N}$ we may apply Lemma~\ref{cor:trace-nlambda-diagonal-lambda} to the collection $\mathfrak{f}_k$ to obtain a new collection of orthonormal vectors $\mathfrak{g}_k$ with $\spans\mathfrak{f}_k = \spans\mathfrak{g}_k$ and the diagonal of $D$ with respect to $\mathfrak{g}_k$ is constantly zero. 
  Thus $D$ has a zero diagonal with respect to the basis $\mathfrak{g} \coloneqq{} \bigcup_k \mathfrak{g}_k$.
\end{proof}

We stated in the introduction that Theorem~\ref{thm:idempotent-zero-diagonal}\ref{item:mthm5} and \ref{item:mthm6} are equivalent for \emph{any} bounded operator, not merely idempotents, which we now prove. 

\begin{proposition}
  \label{prop:sum-iff-abs-sum-diag}
  An operator $T\in B(H)$ has an absolutely summable diagonal in some basis if and only if it has a summable diagonal in some basis.
\end{proposition}

\begin{proof}
  One direction is trivial. 
  For the other direction, suppose that $T\in B(H)$ and $\mathfrak{e} \coloneqq{} \{e_n\}_{n\in\mathbb{N}}$ is a basis with respect to which the corresponding diagonal $\langle d_n \rangle$ is summable with sum $s$. 
  Then there exists a strictly increasing sequence of positive integers $\langle n_k \rangle$ with the property that $\abs{s_{n_k}-s} \le 2^{-k}$, where $s_m$ denotes the partial sum $\sum_{n=1}^m d_j$.

  Since $\sum_{j=1}^{n_1} d_j = s_{n_1}$, by Corollary~\ref{cor:trace-nlambda-diagonal-lambda} there is an orthonormal set $\{b_j\}_{j=1}^{n_1}$ for which $\spans\{b_j\}_{j=1}^{n_1} = \spans\{e_j\}_{j=1}^{n_1}$ and $(Tb_j,b_j) = \nicefrac{s_{n_1}}{n_1}$.
Similarly for each $k\in\mathbb{N}$, because $\sum_{j=n_k+1}^{n_{k+1}} d_j = s_{n_{k+1}}-s_{n_k}$ there is an orthonormal set $\{b_j\}_{j=n_k+1}^{n_{k+1}}$ for which $\spans\{b_j\}_{j=n_k+1}^{n_{k+1}}=\spans\{e_j\}_{j=n_k+1}^{n_{k+1}}$ and $(Tb_j,b_j) = \nicefrac{(s_{n_{k+1}}-s_{n_k})}{n_{k+1}-n_k}$.
Thus $\mathfrak{b} \coloneqq{} \{b_j\}_{j=1}^\infty$ is a  basis since $\spans\mathfrak{b} = \spans\mathfrak{e}$. 
For convenience of notation, set $n_0=0=s_{n_0}$. 
Then with respect to the basis $\mathfrak{b}$, the diagonal sequence is absolutely summable since 
\begin{align*}
\sum_{j=1}^\infty \abs{(Tb_j,b_j)} &= \sum_{k=0}^\infty \sum_{j=n_k+1}^{n_{k+1}} \abs{(Tb_j,b_j)} \\ 
&= \sum_{k=0}^\infty \sum_{j=n_k+1}^{n_{k+1}} \frac{\sabs{s_{n_{k+1}}-s_{n_k}}}{n_{k+1}-n_k} \\
&= \sum_{k=0}^\infty \abs{s_{n_{k+1}}-s_{n_k}} \\
&= \abs{s_{n_1}} + \sum_{k=1}^\infty \left(\abs{s_{n_{k+1}}-s} + \abs{s-s_{n_k}}\right)  \\
&\le \abs{s_{n_1}} + \sum_{k=1}^\infty \left(2^{-(k+1)} + 2^{-k}\right) \\
&= \abs{s_{n_1}} + \frac{3}{2}. \qedhere
\end{align*}
\end{proof}

\section{DIAGONALS OF THE CLASS OF IDEMPOTENTS AND APPLICATIONS}

In this section we investigate Jasper's initial frame theory problem concerning dual frame pairs via its equivalent operator-theoretic formulation:

\begin{problem}
  Characterize the diagonals of the class of idempotent operators.
\end{problem}

\noindent In particular, we prove that every bounded sequence appears as the diagonal of some idempotent  (Theorem~\ref{thm:idempotent-diagonals-ell-infty}). 
We prove this result in stages. First we consider diagonals of idempotents in $M_2(\mathbb{C})$ (Lemma~\ref{lem:2x2-lemma}). Then we give a direct sum construction of an idempotent with constant diagonal (Proposition~\ref{prop:idempotent-constant-diagonal}). From this we show that any bounded sequence with at least one value repeated infinitely many times appears as the diagonal of some idempotent (Proposition~\ref{prop:infinite-multiplicity-diagonal}). And we conclude by showing that we may obtain any bounded sequence as the diagonal of an idempotent. 

The following technical lemma is a trivial corollary of Theorem~\ref{thm:idempotent-matrix-diagonals} except for its norm bound which we require for the forthcoming results. 

\begin{lemma}
  \label{lem:2x2-lemma}
  If $d\in\mathbb{C}$, then there is a $2\times 2$ idempotent $D\in M_2(\mathbb{C})$ with norm $\snorm{D} \le 6\abs{d}+4$ which takes the values $3d-1,-3d+2$ on its diagonal. 
\end{lemma}

\begin{proof}
  Start with the idempotent 
  \begin{equation*}
  D_z =
  \begin{pmatrix}
    1 & 0 \\
    z & 0 \\
  \end{pmatrix},
  \end{equation*}
  with $z\in\mathbb{C}$ to be chosen later. 
  Conjugating by the (unitary) rotation matrix $R_{\nicefrac{\pi}{4}}$, one obtains
  \begin{align*}
    R_{\nicefrac{\pi}{4}}D_zR_{-\nicefrac{\pi}{4}} &= 
    \begin{pmatrix}
      \frac{1}{\sqrt{2}} & -\frac{1}{\sqrt{2}} \\
      \frac{1}{\sqrt{2}} & \frac{1}{\sqrt{2}} \\
    \end{pmatrix}
    \begin{pmatrix}
      1 & 0 \\
      z & 0 \\
    \end{pmatrix}
    \begin{pmatrix}
      \frac{1}{\sqrt{2}} & \frac{1}{\sqrt{2}} \\
      -\frac{1}{\sqrt{2}} & \frac{1}{\sqrt{2}} \\
    \end{pmatrix}
    =
    \begin{pmatrix}
      \frac{1-z}{2} & \frac{1-z}{2} \\
      \frac{1+z}{2} & \frac{1+z}{2} \\
    \end{pmatrix}
  \end{align*}
  Choosing $z = 6d-3$ gives the correct diagonal values. 
  Furthermore, $\snorm{D_z} \le 1 + \abs{z} \le 6d + 4$. 
  Then $D = R_{\nicefrac{\pi}{4}}D_zR_{-\nicefrac{\pi}{4}}$ gives our required idempotent. 
\end{proof}

In the next proposition we exhibit an idempotent with constant diagonal $d$. 
The idea is to take an infinite direct sum of the $2\times 2$ matrix $D$ from Lemma~\ref{lem:2x2-lemma} (whose diagonal entries $d_1,d_2$ satisfy $2d_1 + d_2 = 3d$), regroup the diagonal entries and apply Corollary~\ref{cor:trace-nlambda-diagonal-lambda} repeatedly.

\begin{proposition}
  \label{prop:idempotent-constant-diagonal}
  Given $d\in\mathbb{C}$, there is an idempotent $D_d\in B(H)$ with norm $\norm{D_d} \le 6\abs{d}+4$ with constant diagonal $d$ in some basis. 
\end{proposition}

\begin{proof}
  Let $D'$ be the $2\times 2$ idempotent matrix obtained from Lemma~\ref{lem:2x2-lemma} and set $D = \bigoplus_{i=1}^\infty D'$. 
  Then the diagonal of $D$ consists of the values $d_1 = 3d-1$ and $d_2 = -3d+2$, each repeated infinitely many times. 
  With respect to the basis $\mathfrak{e} \coloneqq{} \{e_j\}_{j\in\mathbb{N}}$, the diagonal entries are
  \begin{equation*}
  (De_j,e_j) = 
  \begin{cases}
    d_1 & \text{if $j$ is odd,} \\
    d_2 & \text{if $j$ is even,} \\
  \end{cases}
  \end{equation*}
  and these diagonal entries satisfy $2d_1 + d_2 = 3d$. 
  Let $\pi$ be any permutation of $\mathbb{N}$ which sends $2\mathbb{N}$ onto $3\mathbb{N}$ (i.e., maps the even positive integers to positive multiples of three).
  Create a new basis $\mathfrak{f}:=\{f_j\}_{j\in\mathbb{N}}$ by $f_j := e_{\pi^{-1}(j)}$. 
  Then we have 
  \begin{equation*}
  (Df_j,f_j) = (De_{\pi^{-1}(j)},e_{\pi^{-1}(j)}) = 
  \begin{cases}
    d_1 & \text{if $j\in\mathbb{N}\setminus 3\mathbb{N}$,} \\
    d_2 & \text{if $j\in 3\mathbb{N}$.} \\
  \end{cases}
  \end{equation*}
  For each $j\in 3\mathbb{N}$, the sum of the diagonal entries corresponding to $f_{j-2},f_{j-1},f_j$ is $2d_1+d_2 = 3d$.
  Thus for each $j\in 3\mathbb{N}$ we may apply Corollary~\ref{cor:trace-nlambda-diagonal-lambda} to obtain new orthonormal vectors $g_{j-2},g_{j-1},g_j$ with $\spans\{f_{j-2},f_{j-1},f_j\} = \spans\{g_{j-2},g_{j-1},g_j\}$ (hence $\mathfrak{g}:=\{g_k\}_{k\in\mathbb{N}}$ is a basis) and $(Dg_k,g_k) = d$ for any $k\in\mathbb{N}$. 
  Taking $D_d:= D$ with respect to the basis $\mathfrak{g}$ is the required idempotent.
\end{proof}

Using Proposition~\ref{prop:idempotent-constant-diagonal} we will now prove that any bounded sequence with at least one value repeated infinitely many times appears as the diagonal of some idempotent. 

\begin{proposition}
  \label{prop:infinite-multiplicity-diagonal}
  Suppose $d := \langle d_j \rangle \in \ell^\infty$ and for some $m$ one has $d_m=d_k$ for infinitely many $k\in\mathbb{N}$.
Then there exists an idempotent $D\in B(H)$ with diagonal $d$ for which $\norm{D} \le 18\norm{d}_\infty+4$. 
\end{proposition}

\begin{proof}
  Observe that the direct sum of idempotents from Proposition~\ref{prop:idempotent-constant-diagonal}:
  \begin{equation*}
  D := \bigoplus_{j=1}^\infty (D_{d_j} \oplus D_{-d_j+2d_m}),
  \end{equation*}
  is a bounded operator whose norm satisfies 
  \begin{align*}
    \norm{D} &= \sup_{j} \{\snorm{D_{d_j}},\snorm{D_{-d_j+2d_m}} \} \\
    &\le \sup_{j} \{6\abs{d_j}+4,6\abs{-d_j+2d_m}+4 \} \\
    &\le 18\snorm{d}_\infty+4. 
  \end{align*}
The idempotent $D$ comes with an associated basis $\mathfrak{e}:=\{e_{i,j,k} \mid i=1,2;\ j,k\in\mathbb{N}\}$ with respect to which the diagonal is 
  \begin{equation*}
  (De_{i,j,k},e_{i,j,k}) = 
  \begin{cases}
    d_j & \text{if $i=1$,} \\
    -d_j+2d_m & \text{if $i=2$.} \\
  \end{cases}
  \end{equation*}
  Create a new basis by the following procedure. 
  Set $f_j := e_{1,j,1}$, so that $(Df_j,f_j) = d_j$. 
  Then for each $j,k\in\mathbb{N}$, apply Corollary~\ref{cor:trace-nlambda-diagonal-lambda} to the pair $e_{1,j,k+1},e_{2,j,k}$ to obtain orthonormal vectors $g_{1,j,k},g_{2,j,k}$ with the same span and corresponding diagonal entries $d_m = \frac{1}{2}(d_j+(-d_j+2d_m))$. 
  Then $\mathfrak{g} := \{f_j\}_{j\in\mathbb{N}} \cup \{g_{i,j,k} \mid i=1,2;\ j,k\in\mathbb{N}\}$ is a basis with diagonal entries $d=\langle d_j \rangle$ (from the $f_j$) along with $d_m$ with infinite multiplicity (from the $g_{i,j,k}$).
  Since the value $d_m$ is repeatedly infinitely many times in the sequence $d$, after a suitable relabeling (permutation of the basis), the diagonal is precisely the sequence $d$. 
\end{proof}

Before we prove our main result for this section we need  Fan's quantitative version of the Toeplitz--Hausdorff Theorem on the convexity of the numerical range. 
As a matter of notation, throughout the remainder of this paper we will use $[a,b]$ to denote the complex line segment joining $a,b \in \mathbb{C}$. 
Then each $d \in [a,b]$ has a \emph{convexity coefficient} $\lambda$ defined by $d = \lambda a + (1-\lambda)b$ for $0 \le \lambda \le 1$, with the convention that $\lambda = 0$ when $a=b$. 
Equivalently, $\lambda = \frac{b-d}{b-a}$ if $b \not= a$ and $\lambda = 0$ if $b = a$. 

\begin{lemma}[\protect{\cite[Lemma 3]{Fan-1984-TotAMS}}]
  \label{lem:quantitative-toeplitz-hausdorff}
  Let 
  \begin{equation*}
  A = 
  \begin{pmatrix}
    d_1       & {\Huge *} \\
    {\Huge *} & d_2       \\
  \end{pmatrix}
  \in M_2(\mathbb{C})
  \end{equation*}
  be a matrix with respect to the basis $\{e_1,e_2\}$ and let $d \in [d_1,d_2]$ with convexity coefficient $\lambda$.
Then there exists a basis $\{b,f\}$ for which $(Af,f)=d$, $(Ab,b)=d_1+d_2-d$ and $\abs{(e_1,f)}^2 \le \lambda$. 
\end{lemma}

We bootstrap this lemma to modify diagonals in an interesting useful way in \autoref{lem:bootstrap-fan}. 
A main tool is to use the following lemma in the case all convexity coefficients $\lambda_n \equiv \nicefrac{1}{2}$ to prove both \autoref{thm:idempotent-diagonals-ell-infty} and \autoref{thm:finite-rank-idempotent-diagonals}. 

\begin{lemma}
  \label{lem:bootstrap-fan}
  Suppose that $T$ is an operator and $\mathfrak{e} = \{e_n\}_{n=0}^{\infty}$ an orthonormal set. 
  Let $r_n \coloneqq{} (Te_n,e_n)$ and suppose $\angles{d_n}_{n=1}^{\infty}$ is a sequence such that for $n \ge 1$, $d_n \in [d_{n-1},r_n]$ with convexity coefficient $\lambda_n$ and where $d_0 \coloneqq{} r_0$.
  Then there is an orthonormal set $\mathfrak{b} = \{b_n\}_{n=1}^{\infty}$ for which $(Tb_n,b_n) = r_n + d_{n-1} - d_n$. 
  Moreover, if $\prod_{i=n}^{\infty} \lambda_i = 0$ for all $n \in \mathbb{N}$, then $\spans \mathfrak{b} = \spans \mathfrak{e}$. 
\end{lemma}

\begin{proof}
  Set $f_0 \coloneqq{} e_0$. 
  Since $d_1 \in [d_0,r_1] = [r_0,r_1]$, by \autoref{lem:quantitative-toeplitz-hausdorff} with diagonal entries $r_0,r_1$, there exist orthonormal $f_1,b_1$ for which $\spans\{f_1,b_1\} = \spans\{f_0,e_1\}$ and $(Tf_1,f_1) = d_1$ and $(Tb_1,b_1) = r_1+r_0-d_1 = r_1+d_0-d_1$ and $\abs{(f_1,f_0)}^2 \le \lambda_1$. 

  Iterating this procedure produces an orthonormal set $\mathfrak{b} = \{b_n\}_{n=1}^{\infty}$ and a sequence of unit vectors $\{f_n\}_{n=0}^{\infty}$ satisfying, for each $n\in\mathbb{N}$,
  \begin{enumerate}[label={\textup{(\roman*)}$_n$}, ref={\textup{(\roman*)}}]
  \item $\spans \{f_n,b_n\} = \spans \{f_{n-1},e_{n}\}$; \label{item:2x2-span-3}
  \item $(Tf_n,f_n) = d_n$ and $(Tb_n,b_n) = r_n+d_{n-1}-d_n$; \label{item:desired-diagonal-3}
  \item $\abs{(f_n,f_{n-1})}^2 \le \lambda_n$; \label{item:small-inner-product-3}
  \item $\{b_1,\ldots,b_n,f_n\}$ is an orthonormal set; \label{item:orthonormality-3} 
  \item $\spans \{b_1,\ldots,b_n,f_n\} = \spans \{e_0,\ldots,e_n\}$. \label{item:same-span-3}
  \end{enumerate}
  We prove this via induction.
  The case $n=1$ is handled in the first paragraph. 
  
  Suppose that \ref{item:2x2-span-3}$_n$--\ref{item:same-span-3}$_n$ hold for some fixed $n \in \mathbb{N}$. 
  Then by hypothesis and \ref{item:desired-diagonal-3}$_n$ one has $d_{n+1} \in [d_n,r_{n+1}] = [(Tf_n,f_n),(Te_{n+1},e_{n+1})]$, so we may apply \autoref{lem:quantitative-toeplitz-hausdorff} to obtain orthonormal $f_{n+1},b_{n+1}$ for which \ref{item:2x2-span-3}$_{n+1}$--\ref{item:small-inner-product-3}$_{n+1}$ hold.
  By \ref{item:orthonormality-3}$_n$ we know that $f_n$ is orthogonal to $\spans \{b_1,\ldots,b_n\}$, and by \ref{item:same-span-3}$_n$ we know $e_{n+1}$ is orthogonal to $\spans \{b_1,\ldots,b_n\}$.
  Thus we obtain $\spans \{b_1,\ldots,b_n\}$ is orthogonal to $\spans \{f_n,e_{n+1}\} = \spans \{b_{n+1},f_{n+1}\}$ by \ref{item:2x2-span-3}$_{n+1}$, thereby establishing \ref{item:orthonormality-3}$_{n+1}$. 
  Finally, by \ref{item:2x2-span-3}$_{n+1}$ and \ref{item:same-span-3}$_n$ we find
  \begin{equation*}
    \spans \{b_1,\ldots,b_{n_+1},f_{n+1}\} = \spans\{b_1,\ldots,b_n,f_n,e_{n+1}\} = \spans \{e_0,\ldots,e_{n+1}\},
  \end{equation*}
  proving \ref{item:same-span-3}$_{n+1}$.
  Hence by induction we have shown \ref{item:2x2-span-3}--\ref{item:same-span-3} for all $n \in \mathbb{N}$.

  Suppose now that $\prod_{i=n}^{\infty} \lambda_i = 0$ for each $n\in\mathbb{N}$.
  Let $P_n$ be the projection on $\{b_1,\ldots,b_n\}$ and let $P$ be the projection onto $\spans \mathfrak{e}$.
  Observe $\spans \mathfrak{b} \subseteq \spans \mathfrak{e}$ by \autoref{item:same-span-3}, and so to prove $\spans\mathfrak{b} = \spans \mathfrak{e}$ it suffices to show that $(P-P_{n+k})e_n\to 0$ in norm as $k\to\infty$ for each $n\in\mathbb{Z}_{\ge 0}$. 
    
  Since $f_j \in \{f_{j-1},e_j\}$ for all $j \in \mathbb{N}$ by \ref{item:2x2-span-3}, one has
  \begin{equation}
    \label{eq:inner-product-split}
    \begin{aligned}
      (e_n,f_{n+k}) &= \Big(e_n,(f_{n+k},f_{n+k-1})f_{n+k-1}+(f_{n+k},e_{n+k})e_{n+k}\Big) \\ 
      &= (e_n,f_{n+k-1})\cdot (f_{n+k},f_{n+k-1}),
    \end{aligned}
  \end{equation}
  and from \ref{item:orthonormality-3}--\ref{item:same-span-3},  $P-P_{n+k}$ is the projection onto $\spans \{f_{n+k},e_{n+k+1},e_{n+k+2},\ldots\}$. 
  This, along with \ref{item:small-inner-product-3} and repeated use of \eqref{eq:inner-product-split} proves
  \begin{align*}
    \snorm{(P-P_{n+k})e_n}^2 & = \abs{(e_n,f_{n+k-1})}^2 \cdot \abs{(f_{n+k},f_{n+k-1})}^2                         \\
                               & = \abs{(e_n,f_n)}^2 \cdot \prod_{i=1}^k \abs{(f_{n+i},f_{n+i-1})}^2               
                                \le \abs{(e_n,f_n)}^2 \cdot \prod_{i=n+1}^{n+k} \lambda_i. 
  \end{align*}
  As $k\to\infty$ the latter product converges to zero by hypothesis.
\end{proof}

Our main result for this section characterizes the diagonals of the class of idempotents to be $\ell^{\infty}$. This, according to Jasper, also characterizes all inner products of dual frame pairs. 

\secondmaintheorem

\begin{proof}
  Let $\mathbb{N} = \bigsqcup_{j\in\mathbb{N}} \mathbb{N}_j$ be any partition of $\mathbb{N}$ such that each $\mathbb{N}_j$ is infinite. 
  Let $\phi_j:\mathbb{N}\to\mathbb{N}_j$ be any bijection. 
  Then for each $j$ define $d_{j,n} \coloneqq{} d_{\phi_j(n)}$;
  in this way we partition the desired sequence into infinitely many infinite sequences. 
  By Proposition~\ref{prop:infinite-multiplicity-diagonal} there is an idempotent $D\in B(H)$ and a basis $\mathfrak{e} = \bigsqcup_j \mathfrak{e}_j$ where $\mathfrak{e}_j \coloneqq{} \{e_{j,n}\}_{n\in\mathbb{Z}_{\ge 0}}$ for which
  \begin{equation*} 
  d_{j,0} \coloneqq{} 0 = (De_{j,0},e_{j,0}) \quad\text{and}\quad 2d_{j,n}-d_{j,n-1} = (De_{j,n},e_{j,n})\ \text{for $n\in\mathbb{N}$}.
  \end{equation*}
  In the above we have assigned $d_{j,0}=0$, and since there are infinitely many zeros, we can apply \autoref{prop:infinite-multiplicity-diagonal}. 
  Note however that $d_{j,0}$ bears no relation to the sequence $\langle d_n \rangle$, unlike $d_{j,n}$ when $n>0$. 

  The remainder of the argument is independent of $j$.
  For each $j$ we will employ a judicious use of \autoref{lem:bootstrap-fan}. 
  Our initial orthonormal set will be $\mathfrak{e}_j$ with diagonal entries $r_{j,n} = (De_{j,n},e_{j,n}) = 2d_{j,n}-d_{j,n-1}$. 
  We then note that $d_{j,n} \in [d_{j,n-1},2d_{j,n}-d_{j,n-1}]$ with convexity coefficient $\lambda_{j,n} = \nicefrac{1}{2}$ since $d_{j,n} = \frac{1}{2}(d_{j,n-1} + (2d_{j,n}-d_{j,n-1}))$. 
  Thus for any $n \in \mathbb{N}$. 
  \begin{equation*}
    \prod_{i=n+1}^{\infty} \lambda_{j,i} = \prod_{i = n+1}^{\infty} \frac{1}{2} = 0.
  \end{equation*}
  By \autoref{lem:bootstrap-fan} there exists an orthonormal set $\mathfrak{b}_j = \{b_{j,n}\}_{n=1}^{\infty}$ for which
  \begin{equation*}
    (Db_{j,n},b_{j,n}) = r_{j,n} + d_{j,n-1} - d_{j,n} = (2d_{j,n} - d_{j,n-1}) + d_{j,n-1} - d_{j,n}  = d_{j,n},
  \end{equation*}
  and $\spans \mathfrak{b}_j = \spans \mathfrak{e}_j$.
  Thus $\mathfrak{b} \coloneqq{} \bigcup_j \mathfrak{b}_j$ is a basis because $\mathfrak{e} = \bigcup_j \mathfrak{e}_j$ is a basis.
  With respect to the basis $\mathfrak{b}$ the idempotent $D$ has diagonal $\angles{d_{j,n}}$ which is precisely $\angles{d_n}$ after a suitable relabeling.
\end{proof}

\section{DIAGONALS OF THE CLASS OF FINITE RANK IDEMPOTENTS}

Recall that Lemma~\ref{lem:idempotent-decomp} is valid for both finite and infinite dimensional $H$. 
As a result, for $D\in M_n(\mathbb{C})$ with $0\not=D\not=I$, $\trace D = \rank D \in \{1,\ldots,n-1\}$. 
Theorem~\ref{thm:idempotent-matrix-diagonals} shows that this trace condition is the only restriction for a given sequence to be the diagonal of a nonzero non-identity idempotent matrix. 
Because not all idempotent operators $D\in B(H)$ ($H$ infinite dimensional) are trace-class, it is unnatural to expect there to be any sort of trace restriction on the diagonals of idempotent operators in $B(H)$. 
In this light, Theorem~\ref{thm:idempotent-diagonals-ell-infty} is naturally expected: if the only restriction in the $n\times n$ matrix case was the trace, there should be no restrictions in $B(H)$. 

However, there is another perfectly reasonable class to consider: the trace-class idempotents. 
Again, Lemma~\ref{lem:idempotent-decomp} ensures that trace-class idempotents are actually finite rank idempotents. 
The restriction that $\trace D = \rank D \in \mathbb{N}$ is still applicable for finite rank idempotents $D\in B(H)$. 
In this section we prove that, as for $M_n(\mathbb{C})$, this trace condition is the only restriction for an $\ell^1$ (absolutely summable) sequence to be the diagonal of a finite rank idempotent, which is Theorem~\ref{thm:finite-rank-idempotent-diagonals} below.

A corollary of the next lemma verifies Theorem~\ref{thm:finite-rank-idempotent-diagonals} when restricted to rank-one idempotents. 
That is, the diagonals of the class of rank-one idempotents are precisely those absolutely summable sequences which sum to one. 

\begin{lemma}
  \label{lem:T^2=Tr(T)T}
  If $T\in B(H)$ is a rank-one operator then $T^2 = \trace(T)T$, hence $T$ is idempotent if and only if $\trace T = 1$.
\end{lemma}

\begin{proof}
  We may write any rank-one operator as an infinite matrix with entries $a_ib_j$ where $\langle a_i \rangle,\langle b_j \rangle\in\ell^2$. 
  Since the trace is independent of the choice of basis, $\trace T = \sum_{k=1}^\infty a_kb_k$.
  Finally, 
  \begin{equation*}
  T^2 = \left( \sum_{k=1}^\infty (a_ib_k)(a_kb_j) \right) = \left( a_i\left(\sum_{k=1}^\infty a_kb_k\right) b_j\right) = \trace(T) (a_ib_j) = \trace(T)T. 
  \end{equation*}
  Another proof which is less, but not entirely, coordinate free: since $T$ is rank-one, there are $x,y \in H$ for which $Tz = (z,x)y$. 
  By expanding $T$ in a basis for $H$ which contains $\nicefrac{y}{\snorm{y}}$, it is clear that $\trace T = (y,x)$. 
  Thus 
  \begin{equation*}
    T^2z = T(z,x)y = (z,x)(y,x)y = (y,x)Tz = \trace(T)Tz. \qedhere
  \end{equation*}
\end{proof}

\begin{corollary}
  \label{cor:rank-1-idempotent-diags}
  An absolutely summable sequence $\langle d_j \rangle\in\ell^1$ is the diagonal of some rank-one idempotent $D$ if and only if $\sum_j d_j = 1$. 
\end{corollary}

\begin{proof}
  One direction is trivial since $\sum_j d_j = \trace D = \rank D = 1$ by Lemma~\ref{lem:T^2=Tr(T)T}.
  
  For the other direction, let $\langle d_j \rangle \in\ell^1$ be any absolutely summable sequence which sums to one.  
  Write $d_j = r_j\mathrm{e}^{\mathrm{i}\theta_j}$ with $r_j\ge 0$ and $j\in\mathbb{R}$.  
  Then define $\sqrt{d}\in\ell^2$ as $(\sqrt{d})_j \coloneqq{} \sqrt{r_j}\mathrm{e}^{\nicefrac{\mathrm{i}\theta_j}{2}}$.  
  Then define $D = \left((\sqrt{d})_i(\sqrt{d})_j\right) = \sqrt{d}\otimes\sqrt{d}$. 
  By Lemma~\ref{lem:T^2=Tr(T)T}, $D$ is idempotent since its diagonal is $\langle d_n \rangle$ which sums to one.
\end{proof}

We now prove Theorem~\ref{thm:finite-rank-idempotent-diagonals} by two distinct methods. 
The first uses Theorem~\ref{thm:idempotent-matrix-diagonals}, Corollary~\ref{cor:rank-1-idempotent-diags}, and \autoref{lem:bootstrap-fan}.
The second proof is an inductive argument analogous to the proof of Theorem~\ref{thm:idempotent-matrix-diagonals} by Giol, Kovalev, Larson, Nguyen and Tener in \cite{GKL+-2011-OaM}.
It uses Corollary~\ref{cor:rank-1-idempotent-diags} as the base case and exploits the fact that the class of finite rank idempotents is similarity invariant.

\thirdmaintheorem

\begin{proof}[Proof using \autoref{lem:bootstrap-fan}]
  Lemma~\ref{lem:idempotent-decomp} makes this sum condition obviously necessary, so sufficiency is all that is needed. 
Let $d \coloneqq{} \langle d_n \rangle\in\ell^1$ be an absolutely summable sequence whose sum $\sum_n d_n = m$ is a positive integer. 
  If $m=1$, then $\langle d_n \rangle$ is the diagonal of a rank-one idempotent by Corollary~\ref{cor:rank-1-idempotent-diags}. 
  So suppose $m>1$, in which case $m-1\in\mathbb{N}$. 
  Set $d'_m \coloneqq{} (m-1) - \sum_{n=1}^{m-1} d_n$. 
  By Theorem~\ref{thm:idempotent-matrix-diagonals}, there is an idempotent matrix $D_1 \in M_m(\mathbb{C})$ with diagonal $d^{(1)} \coloneqq{} \langle d_1,\ldots,d_{m-1},d'_m \rangle$. 
  Now consider the sequence $d^{(2)} \coloneqq{} \langle 2d_m-d'_m,2d_{m+1}-d_m,2d_{m+2}-d_{m+1},\ldots \rangle$. 
  It is clear that $d^{(2)}\in\ell^1$ since $d\in\ell^1$. 
  Furthermore, 
  \begin{equation*}
  \sum_{n=1}^\infty d^{(2)}_n = 2d_m-d'_m + \sum_{n=m}^\infty (2d_{n+1}-d_n) = -d'_m + \sum_{n=m}^\infty d_n = \sum_{n=1}^\infty d_n - (m-1) = 1. 
  \end{equation*}
  Therefore, by Corollary~\ref{cor:rank-1-idempotent-diags}, there is a rank-one idempotent $D_2$ with diagonal sequence $d^{(2)}$. 
  Defining $D = D_1 \oplus D_2$, we find that $D$ has a basis $\mathfrak{e} \coloneqq\{e_n\}_{n\in\mathbb{N}}$ in which its diagonal is 
  \begin{equation*}
  \langle d_1, \ldots, d_{m-1}, d'_m, 2d_m-d'_m, 2d_{m+1}-d_m, 2d_{m+2}-d_{m+1}, \ldots \rangle.
  \end{equation*}
  That is, $(De_n,e_n) = d_n$ for $1\le n<m$; $(De_m,e_m) = d'_m$; $(De_{m+1},e_{m+1}) = 2d_m-d'_m$; and $(De_n,e_n) = 2d_{n-1}-d_{n-2}$ for $n>m+1$. 

  We will now apply \autoref{lem:bootstrap-fan} to the orthonormal set $\{e_m,e_{m+1},\ldots\}$. 
  So 
  \begin{equation*}
    r_n \coloneqq{} (De_n,e_n) =
    \begin{cases}
      d'_m               & \text{if $n=m$}   \\
      2d_m - d'_m        & \text{if $n=m+1$} \\
      2d_{n-1} - d_{n-2} & \text{if $n>m+1$} \\
    \end{cases}
  \end{equation*}
  Since $d_m \in [r_m,r_{m+1}]$ and $d_n \in [d_{n-1},r_{n+1}]$ for $n>m$ (with convexity coefficients all $\lambda_n \equiv \nicefrac{1}{2}$), after a suitable relabeling of the sequences involved ($r_n \mapsto r_{n-m}; d_n \mapsto d_{n-m+1}$) we may apply \autoref{lem:bootstrap-fan} to obtain an orthonormal set $\{b_m,b_{m+1},\ldots\}$ satisfying
  \begin{equation*}
    (Db_n,b_n) = 
    \begin{cases}
      r_{m+1}+r_m-d_m & \text{if $n=m$} \\
      r_{n+1}+d_{n-1}-d_n & \text{if $n>m$} \\
    \end{cases}
    \Bigg\}
    = d_n.
  \end{equation*}
  Moreover, by \autoref{lem:bootstrap-fan}, since the convexity coefficients are $\lambda_n = \nicefrac{1}{2}$, we have $\spans\{b_n\}_{n=m}^{\infty} = \spans\{e_n\}_{n=m}^{\infty}$. 
  Setting $b_n \coloneqq{} e_n$ for $n<m$, we find that $\mathfrak{b} = \{b_n\}_{n=1}^{\infty}$ is a basis with respect to which $D$ has diagonal $\angles{d_n}$.
\end{proof}

\begin{proof}[Proof by induction using techniques from \cite{GKL+-2011-OaM}]
We proceed by induction on the sum $m\coloneqq{} \sum_{n=1}^\infty d_n$ where $\langle d_n \rangle\in\ell^1$ is an absolutely summable sequence whose sum is a positive integer. 
The base case $m=1$ is handled by Corollary~\ref{cor:rank-1-idempotent-diags}. 

Now suppose $m>1$ and for any absolutely summable sequence whose sum is $m-1$, there is a finite rank idempotent with that sequence on its diagonal. 
By possibly permuting the sequence $d_n$, we may assume without loss of generality that $d_1+d_2\not=2$. 
Since $\sum_{n=1}^\infty d_n = m$, then $(d_1+d_2-1)+\sum_{n=3}^\infty d_n = m-1$. 
So by the induction hypothesis there exists a finite rank (in fact, rank-$(m-1)$) idempotent $\tilde{D}$ with diagonal sequence $\langle d_1+d_2-1,d_3,d_4,\ldots \rangle$. 
Then consider the rank-$m$ operator 
\begin{equation*}
D' =
\begin{pmatrix}
  1                  & 0_{1\times\infty} \\
  0_{\infty\times 1} & \tilde{D}         \\
\end{pmatrix}, 
\end{equation*}
which is obviously idempotent. 
With respect to the basis $\mathfrak{e} = \{e_j\}_{j=1}^\infty$, $D'$ has diagonal $\langle 1,d_1+d_2-1,d_3,d_4,\ldots \rangle$.
Then consider the invertible $S$ which is the identity on $\spans\{e_j\}_{j=3}^\infty$ and whose compression to $\spans\{e_1,e_2\}$ has the matrix representation
\begin{equation*}
\begin{pmatrix}
  \lambda            & \lambda-1           \\
  1                  & 1                   \\
\end{pmatrix},
\end{equation*}
where $\lambda \coloneqq{} \nicefrac{(d_2-1)}{(d_1+d_2-2)}$.
Conjugating $D'$ by $S$ produces an idempotent $D \coloneqq{} SD'S^{-1}$ whose diagonal with respect to $\mathfrak{e}$ is precisely the sequence $d$. 

The reader should note that although conjugating by a similarity can be viewed as changing the \emph{linear} basis  (as opposed to conjugating by a unitary which changes the \emph{orthonormal} basis) we are not using the similarity in this context. 
Instead, we only use the similarity to produce a new idempotent $D$ (which still has finite rank) and has the desired diagonal with respect to the \emph{orthonormal} basis~$\mathfrak{e}$.
\end{proof}

\section*{ACKNOWLEDGMENTS}

The authors are indebted to John Jasper who at GPOTS 2013 alerted us to his frame theory and equivalent operator theory problems and their history. 
And special thanks to Daniel Belti\c{t}\u{a} for the reference to the work of Fan, Fong and Herrero \cite{Fan-1984-TotAMS}, \cite{FF-1994-PotAMS}, \cite{FFH-1987-PotAMS}.

Jean-Christophe Bourin deserves special mention for communicating to us that our \autoref{thm:idempotent-diagonals-ell-infty} is a corollary of his pinching theorem \cite[Theorem 2.1]{Bou-2003-JOT}. 
In particular, to apply his theorem one only needs to produce idempotents whose essential spectra contain disks centered at the origin.  
This is especially interesting because his methods of proof are completely different from ours.

The first author was supported by the Charles Phelps Taft Dissertation Fellowship. 
The second author was partially supported by the Simons Foundation Collaboration Grant for Mathematicians \#245014 and the Charles Phelps Taft Research Center.


\end{document}